\documentclass{article}
\usepackage{amsfonts}
\usepackage{amsmath}
\usepackage{calc}
\usepackage{graphicx}
\usepackage{fancyhdr}
\usepackage{amsmath}
\usepackage{amsfonts}
\usepackage{amssymb}
\usepackage{verbatim} % for comment environment
\usepackage{epic}
\usepackage{amsthm} 
\usepackage[colorlinks=true, pdfstartview=FitV, linkcolor=black, 
      citecolor=black, urlcolor=black]{hyperref}
\usepackage[usenames,dvipsnames]{pstricks}
\usepackage{epsfig}
\usepackage{pst-grad}
\usepackage{pst-plot}
\usepackage{float}

\numberwithin{equation}{section}
\renewcommand{\thesection}{\arabic{section}}
\renewcommand{\theequation}{\thesection.\arabic{equation}}
\renewcommand{\thefigure}{\thesection.\arabic{figure}}

\newtheorem{theorem}[equation]{Theorem}
\newtheorem{lemma}[equation]{Lemma}
\newtheorem{proposition}[equation]{Proposition}
\newtheorem{corollary}[equation]{Corollary}
\newtheorem{remark}[equation]{Remark}

\newtheorem{definition}[equation]{Definition}

\newenvironment{my_enumerate}
{\begin{enumerate}
  \setlength{\itemsep}{0pt}
  \setlength{\parskip}{0pt}
  \setlength{\parsep}{0pt}}
{\end{enumerate}}

\newcommand{\mult}{\operatorname{mult}}
\newcommand{\Sec}{\operatorname{Sec}}
\newcommand{\bbC}{{\mathbb{C}}}
\newcommand{\bbP}{{\mathbb{P}}}
\newcommand{\bbR}{{\mathbb{R}}}
\newcommand{\bbF}{{\mathbb{F}}}
\newcommand{\cD}{{\mathcal{D}}}
\newcommand{\cL}{{\mathcal{L}}}
\newcommand{\cE}{{\mathcal{E}}}
\newcommand{\cH}{{\mathcal{H}}}
\newcommand{\cX}{{\mathcal{X}}}

\newcommand{\cO}{{\mathcal{O}}}

\begin{document}
\author{Olivia Dumitrescu}
\title{Plane curves with prescribed triple points: a toric approach}
%\date{\today}
\maketitle

\begin{abstract}
We will use toric degenerations of the projective plane ${{\mathbb{P}}^ 2}$ to give a new proof of the triple points interpolation problems in the projective plane. We also give a complete list of toric surfaces that are useful as components in this degeneration.
\end{abstract}

\section{Introduction}
Let ${\mathcal{L}}_{d}(m_{1},...,m_{r})$ denote the linear system of curves in ${{\mathbb{P}}^ 2}$ of degree $d$, that pass through $r$ points $P_{1},...,P_{r}$, with multiplicity at least $m_{i}$. A natural question would be to compute the projective dimension of the linear system ${\mathcal{L}}$. 
The virtual dimension of $\mathcal{L}$ is
$$\ v
({\mathcal{L}}_{d}(m_1,...,m_r)):=  \binom {d+2} {2}-\sum^{r}_{i=1} \binom {m_{i}+1} {2}-1
$$
and the expected dimension is
$\ e ({\mathcal{L}}):=max\{v ({\mathcal{L}}),-1\}.$
There are some elementary cases for which $\ dim({\mathcal{L}}) \neq \ e ({\mathcal{L}}).$ A linear system for which $dim {\mathcal{L}}> e ({\mathcal{L}})$ is called a special linear system. However if we consider the homogeneous case when all multiplicities are equal $m_1=...=m_r=m$, the linear system (denoted by ${\mathcal{L}}_{d}(m^r)$) is expected to be non-special when $r$ is large enough. In this paper we will only consider the case $m=3$; the virtual dimension becomes $v({\mathcal{L}}_{d}(3^r))=\frac{d(d+3)}{2}-6m-1$. In \cite{CDM07} the authors used a toric degeneration of the Veronese surface into a union of projective planes for the double points interpolation problems i.e. $m=2$. This paper extends the degeneration used in \cite{CDM07} to the triple points interpolation problem.
A triple point in the projective plane imposes six conditions, so in this paper we will classify the toric surfaces $(X, {\mathcal{L}})$ with $h^{0}(X, {\mathcal{L}})=6$ (see Theorem \ref{Classification}). In particular we will analyse the ones for which the linear system becomes empty when imposing a triple point, call them $Y_{i}$. 
We will then use a toric degeneration of the embedded projective space via a linear system  ${\mathcal{L}}$ into a union of planes, quadircs and $r$ disjoint toric surfaces $Y_{i}$. On each surface $Y_{i}$ we will place one triple point and by a semicontinuity argument we will prove the non-speciality of ${\mathcal{L}}$, see
Theorem \ref {triple plane}.\\\\
We remark that this result can be generelized to any dimension i.e. a list of toric varieties becoming empty when imposing a muliplicity $m$ point could be described in a similar way. However the combinatorical degeneration and the construction of the lifting function is not very well understood. 
In  \cite{Len08} T. Lenarcik used an algebraic approach to study triple points interpolation in ${\mathbb{P}}^ {1}\times {\mathbb{P}}^ {1}$; however the list of the algebraic polygons is slightly different than ours and the connection with toric degenerations is not explicit.
Toric degenerations of three dimensional projective space have been used by S. Brannetti to give a new proof of the Alexander-Hirschowitz theorem in dimension three in \cite{Brann08}. Degenerations of  $n$-dimensional projective space have been used by E. Postinghel to give a new proof for the Alexander-Hirschowitz theorem in any dimension in \cite{Postinghel}.

%\indent
%Consider $n=d=r=m_{1}=m_{2}=2$, for example. One can notice that $\ dim ({\mathcal{L}}_{2,2})(2^2)=0$ since the plane system of conics with two general double points consists of a fixed divisor: the unique double line through the two points, so its projective dimension is $0$. This is different than the expected dimension $\ e ({\mathcal{L}}_{2,2})(2^2)=-1$.

\indent

\section{Toric varieties and toric degenerations.}\label{sec:toric}

We recall a few basic facts
about toric degenerations of projective toric varieties. We refer to \cite {hu} and \cite{WF}
for more information on the subject
and to \cite {gath} for relations with tropical geometry.

The datum of a pair $(X,\cL)$,
where $X$ is a projective, $n$--dimensional toric variety
and $\cL$ is a base point free, ample line bundle on $X$,
is equivalent to the datum of
an $n$ dimensional convex polytope ${\mathcal{P}}$ in $\bbR^ n$, determined up to translation and $SL_{n}^{\pm}({\mathbb{Z}})$.
We consider a polytope ${\mathcal{P}}$ and a \emph{subdivision} ${\mathcal{D}}$ of ${\mathcal{P}}$ into convex subpolytopes; i.e. a finite family of $n$ dimensional convex
polytopes whose union is ${\mathcal{P}}$ and such that
any two of them intersect only along a face (which may be empty).
Such a subdivision is called \emph{regular} if there is a convex, positive, linear on each face, function $F$ defined on ${\mathcal{P}}$. Such function $F$ will be called a lifting function. Regular subdivisions correspond to degeneration of toric varieties.

We will now prove a technical lemma that will enable us to easily demonstrate the existence of a lifting function when we need them. 
Let $X$ be a toric surface and ${\mathcal{P}}$ be its associated polygon. 
Consider $L$ a line that separates two disjoint polygons ${\mathcal{P}_1}$ and ${\mathcal{P}_2}$ in ${\mathcal{P}}$ and $X_{1}$ and $X_{2}$ their corresponding toric surfaces such that $L$ does not contain any integer point. 
\begin{lemma} 
The toric variety $X$ degenerates into a union of toric surfaces two of which are $X_{1}$ and $X_{2}$ which are skew.
\end{lemma}

\begin{proof}
We consider the convex piecewise linear function given by 
$$f(x,y,z)=\max \{z, L+z \}$$
Consider the image of the points on the boundary of the polygons $X_{1}$ and $X_{2}$ through $f$. 
Change the function $f$ by interposing the convex hull of the boundary points separated by $L$ between $l_{1}$ and $l_{2}$ (as in the Figure \ref{lift1}).
The function will still be convex and piecewise linear, therefore we get a regular subdivision.
We consider now the toric varieties associated to each 
polygon, and since ${\mathcal{P}_{1}}$ and ${\mathcal{P}_{2}}$ are disjoint, we obtain that two of the toric surfaces that appeared in the degeration, namely $X_1$ and $X_2$ are skew.
\end{proof}

For example, in the picture below we have four polygons, two of which are disjoint. The corresponding degeneration 
will contain four toric varieties, two of them $X_{1}$ and $X_{2}$, being skew.
\begin{figure}[h]
\begin{center}
\scalebox{1} % Change this value to rescale the drawing.
{
\begin{pspicture}(0,-2.0745313)(4.68,2.0445313)
\definecolor{color6}{rgb}{1.0,0.2,0.2}
\definecolor{color137}{rgb}{1.0,0.4,0.4}
\psline[linewidth=0.02cm](0.03,-1.5654688)(3.45,-1.5654688)
\psline[linewidth=0.02cm](0.03,-1.5454688)(0.87,-0.36546874)
\psline[linewidth=0.02cm](0.87,-0.36546874)(4.27,-0.36546874)
\psdots[dotsize=0.06](0.87,-0.38546875)
\psdots[dotsize=0.06](1.27,-0.38546875)
\psdots[dotsize=0.06](1.65,-0.38546875)
\psdots[dotsize=0.06](2.05,-0.38546875)
\psdots[dotsize=0.06](2.45,-0.38546875)
\psdots[dotsize=0.06](2.85,-0.38546875)
\psdots[dotsize=0.06](3.25,-0.38546875)
\psdots[dotsize=0.06](3.63,-0.38546875)
\psline[linewidth=0.02cm](4.25,-0.36546874)(4.47,-0.36546874)
\psline[linewidth=0.02cm](3.43,-1.5654688)(3.67,-1.5654688)
\psline[linewidth=0.02cm](3.65,-1.5654688)(4.47,-0.36546874)
\psdots[dotsize=0.06](4.05,-0.38546875)
\psdots[dotsize=0.06](4.45,-0.36546874)
\psdots[dotsize=0.06](0.85,-0.7654688)
\psdots[dotsize=0.06](1.25,-0.7654688)
\psdots[dotsize=0.06](1.65,-0.7654688)
\psdots[dotsize=0.06](2.05,-0.7654688)
\psdots[dotsize=0.06](2.45,-0.7654688)
\psdots[dotsize=0.06](2.85,-0.7654688)
\psdots[dotsize=0.06](3.25,-0.7654688)
\psdots[dotsize=0.06](3.65,-0.78546876)
\psdots[dotsize=0.06](4.05,-0.7654688)
\psdots[dotsize=0.06](0.45,-1.1654687)
\psdots[dotsize=0.06](0.85,-1.1654687)
\psdots[dotsize=0.06](1.25,-1.1654687)
\psdots[dotsize=0.06](1.65,-1.1654687)
\psdots[dotsize=0.06](2.03,-1.1654687)
\psdots[dotsize=0.06](2.45,-1.1654687)
\psdots[dotsize=0.06](2.85,-1.1654687)
\psdots[dotsize=0.06](3.25,-1.1854688)
\psdots[dotsize=0.06](3.63,-1.1654687)
\psdots[dotsize=0.06](0.03,-1.5654688)
\psdots[dotsize=0.06](0.43,-1.5654688)
\psdots[dotsize=0.06](0.85,-1.5654688)
\psdots[dotsize=0.06](1.25,-1.5654688)
\psdots[dotsize=0.06](1.63,-1.5654688)
\psdots[dotsize=0.06](2.05,-1.5654688)
\psdots[dotsize=0.06](2.45,-1.5654688)
\psdots[dotsize=0.06](2.85,-1.5654688)
\psdots[dotsize=0.06](3.25,-1.5654688)
\psdots[dotsize=0.06](3.63,-1.5654688)
\psline[linewidth=0.0080cm,linecolor=red,linestyle=dashed,dash=0.16cm 0.16cm](1.71,-1.5654688)(3.07,-0.36546874)
\psline[linewidth=0.02cm](0.85,1.4545312)(0.07,0.25453126)
\usefont{T1}{ptm}{m}{n}
\rput(1.4935937,-0.9504688){\footnotesize $X_{1}$}
\usefont{T1}{ptm}{m}{n}
\rput(3.2935936,-0.9504688){\footnotesize $X_{2}$}
\usefont{T1}{ptm}{m}{n}
\rput(1.4635937,0.94953126){\footnotesize $z=0$}
\usefont{T1}{ptm}{m}{n}
\rput(3.5035937,1.1895312){\footnotesize $L+z=0$}
\usefont{T1}{ptm}{m}{n}
\rput(3.1935937,-0.21046875){\footnotesize \color{red}$L=0$}
\psline[linewidth=0.02cm](3.95,0.95453125)(4.67,2.0345314)
\usefont{T1}{ptm}{m}{n}
\rput(3.0135937,1.6695312){\footnotesize \color{color6}$L=0$}
\psline[linewidth=0.0040cm](1.63,-1.5454688)(2.85,-0.40546876)
\psline[linewidth=0.0040cm](2.03,-1.5654688)(3.23,-0.38546875)
\psline[linewidth=0.0040cm](2.87,-0.36546874)(2.03,-1.5854688)
\psline[linewidth=0.02cm](0.85,1.4545312)(2.69,1.4545312)
\psline[linewidth=0.02cm](0.05,0.25453126)(1.51,0.23453125)
\psline[linewidth=0.02cm](2.69,1.4545312)(1.49,0.23453125)
\psline[linewidth=0.02cm](1.51,0.23453125)(2.27,0.39453125)
\psline[linewidth=0.02cm](2.67,1.4545312)(2.27,0.37453124)
\psline[linewidth=0.02cm](2.67,1.4545312)(3.29,1.5345312)
\psline[linewidth=0.02cm](3.29,1.5345312)(2.27,0.39453125)
\psline[linewidth=0.02cm](3.27,1.5145313)(4.65,2.0145311)
\psline[linewidth=0.02cm](2.27,0.39453125)(3.97,0.9345313)
\usefont{T1}{ptm}{m}{n}
\rput(1.2735938,-1.8504688){\footnotesize $l_{1}$}
\usefont{T1}{ptm}{m}{n}
\rput(2.5335937,-1.8704687){\footnotesize $l_{2}$}
\usefont{T1}{ptm}{m}{n}
\rput(1.1935937,0.04953125){\footnotesize $l_{1}$}
\usefont{T1}{ptm}{m}{n}
\rput(2.6135938,0.20953125){\footnotesize $l_{2}$}
\psline[linewidth=0.02cm,linecolor=color137,linestyle=dashed,dash=0.16cm 0.16cm](1.79,0.29453126)(2.95,1.4745313)
\end{pspicture} 
}
\end{center}
\caption{}
\label{lift1}
\end{figure}

It is easy to see how we could iterate this process. Let $M$ be a line cutting the polygon associated to $X_{2}$ and not containing any of its interior points. Then $X_{2}$ degenerates into a union of toric surfaces, two of which are skew, $Y_{2}$ and $Y_{3}$.

In this case, we conclude that $X$ degenerates into nine toric surfaces three of which $X_{1}$, $Y_{2}$ and $Y_{3}$ being skew, as the Figure \ref{lift2} indicates.
\begin{figure}[h]
\begin{center}
\scalebox{1} % Change this value to rescale the drawing.
{
\begin{pspicture}(0,-1.079375)(4.5540624,1.079375)
\definecolor{color245}{rgb}{1.0,0.0,0.2}
\definecolor{color174}{rgb}{1.0,0.2,0.2}
\psline[linewidth=0.02cm](0.03,-0.5703125)(3.45,-0.5703125)
\psline[linewidth=0.02cm](0.03,-0.5503125)(0.87,0.6296875)
\psline[linewidth=0.02cm](0.87,0.6296875)(4.27,0.6296875)
\psdots[dotsize=0.06](0.87,0.6096875)
\psdots[dotsize=0.06](1.27,0.6096875)
\psdots[dotsize=0.06](1.65,0.6096875)
\psdots[dotsize=0.06](2.05,0.6096875)
\psdots[dotsize=0.06](2.45,0.6096875)
\psdots[dotsize=0.06](2.85,0.6096875)
\psdots[dotsize=0.06](3.25,0.6096875)
\psdots[dotsize=0.06](3.63,0.6096875)
\psline[linewidth=0.02cm](4.25,0.6296875)(4.47,0.6296875)
\psline[linewidth=0.02cm](3.43,-0.5703125)(3.67,-0.5703125)
\psline[linewidth=0.02cm](3.65,-0.5703125)(4.47,0.6296875)
\psdots[dotsize=0.06](4.05,0.6096875)
\psdots[dotsize=0.06](4.45,0.6296875)
\psdots[dotsize=0.06](0.85,0.2296875)
\psdots[dotsize=0.06](1.25,0.2296875)
\psdots[dotsize=0.06](1.65,0.2296875)
\psdots[dotsize=0.06](2.05,0.2296875)
\psdots[dotsize=0.06](2.45,0.2296875)
\psdots[dotsize=0.06](2.85,0.2296875)
\psdots[dotsize=0.06](3.25,0.2296875)
\psdots[dotsize=0.06](3.65,0.2096875)
\psdots[dotsize=0.06](4.05,0.2296875)
\psdots[dotsize=0.06](0.45,-0.1703125)
\psdots[dotsize=0.06](0.85,-0.1703125)
\psdots[dotsize=0.06](1.25,-0.1703125)
\psdots[dotsize=0.06](1.65,-0.1703125)
\psdots[dotsize=0.06](2.03,-0.1703125)
\psdots[dotsize=0.06](2.45,-0.1703125)
\psdots[dotsize=0.06](2.85,-0.1703125)
\psdots[dotsize=0.06](3.25,-0.1903125)
\psdots[dotsize=0.06](3.63,-0.1703125)
\psdots[dotsize=0.06](0.03,-0.5703125)
\psdots[dotsize=0.06](0.43,-0.5703125)
\psdots[dotsize=0.06](0.85,-0.5703125)
\psdots[dotsize=0.06](1.25,-0.5703125)
\psdots[dotsize=0.06](1.63,-0.5703125)
\psdots[dotsize=0.06](2.05,-0.5703125)
\psdots[dotsize=0.06](2.45,-0.5703125)
\psdots[dotsize=0.06](2.85,-0.5703125)
\psdots[dotsize=0.06](3.25,-0.5703125)
\psdots[dotsize=0.06](3.63,-0.5703125)
\psline[linewidth=0.0080cm,linecolor=red,linestyle=dashed,dash=0.16cm 0.16cm](1.71,-0.5703125)(3.07,0.6296875)
\usefont{T1}{ptm}{m}{n}
\rput(1.4935937,0.0446875){\footnotesize $X_{1}$}
\usefont{T1}{ptm}{m}{n}
\rput(2.7535937,-0.8753125){\footnotesize $Y_{2}$}
\psline[linewidth=0.0080cm,linecolor=color174,linestyle=dashed,dash=0.16cm 0.16cm](2.65,0.2496875)(3.77,-0.3903125)
\psline[linewidth=0.02cm](2.03,-0.5703125)(2.45,-0.1503125)
\psline[linewidth=0.02cm](2.45,-0.1503125)(3.23,-0.1703125)
\psline[linewidth=0.02cm](3.23,-0.1703125)(3.63,-0.5703125)
\psline[linewidth=0.02cm](2.85,0.2296875)(3.25,0.6296875)
\psline[linewidth=0.02cm](2.85,0.2296875)(3.63,-0.1703125)
\psline[linewidth=0.02cm](3.63,-0.1703125)(4.45,0.6296875)
\usefont{T1}{ptm}{m}{n}
\rput(3.9135938,0.9046875){\footnotesize $Y_{3}$}
\usefont{T1}{ptm}{m}{n}
\rput(1.4235938,-0.8553125){\footnotesize \color{color245}$L$}
\usefont{T1}{ptm}{m}{n}
\rput(4.213594,-0.4953125){\footnotesize \color{color174}$M$}
\psline[linewidth=0.0040cm](2.43,-0.1503125)(2.83,0.2296875)
\psline[linewidth=0.0040cm](2.85,0.2296875)(3.25,-0.1703125)
\psline[linewidth=0.0040cm](3.25,-0.1703125)(3.59,-0.1503125)
\psline[linewidth=0.0040cm](3.65,-0.1503125)(3.63,-0.5703125)
\psline[linewidth=0.0040cm](1.63,-0.5703125)(2.85,0.6096875)
\psline[linewidth=0.0040cm](2.85,0.6496875)(2.03,-0.5503125)
\end{pspicture} 
}

\end{center}
\caption{}
\label{lift2}
\end{figure}

Later on, we will ignore the varieties lying in between the disjoint ones; they are only important for the degeneration and not for the analysis itself.

\section{The Classification of polygons.}

We recall that the group $SL_2^{\pm}({\mathbb{Z}})$ acts on the column vectors of ${\mathbb{R}}^ 2$ by left multiplication. This induces an action of $SL_2^{\pm}({\mathbb{Z}})$ on the set of convex polygons ${\mathcal{P}}$ by acting on its enclosed points ($SL_2^{-}({\mathbb{Z}})$ corresponds to orientation reversing lattice equivalences). Obviously,
${\mathbb{Z}^{2}}$ acts on vectors of ${\mathbb{R}}^ 2$ by translation.

Next, we will classify all convex polygons enclosing six lattice integer points modulo the actions described above.
We first start with a definition.

\begin{definition} We say the polygon ${\mathcal{P}}$ is equivalent to one in \it{standard position} if
\begin{my_enumerate}
\item It contains $O=(0,0)$ as a vertex
\item $OS$ is a vertex where $S=(0,m)$ and $m$ is the largest edge length
\item $OP$ is an edge where $P=(p,q)$ and $0\leq p <q$
\end{my_enumerate} 
\end{definition}

\begin{remark} Every polygon has a standard position.
\end{remark}

Indeed, we first choose the longest edge and then we translate one of its vertices to the origin. We will now rotate the polygon to put the longest edge on the positive side of the $x$ axis and then we shift it such that the adjacent edge lies in the upper half of the first quadrant. 
Indeed, if $OP$ is an edge with $P=(s,q)$ and $s\geq q$; then $s=mq+p$ for $0\leq p<q$ so we shift left by $m$.
We will call this procedure normalization.\\

It is easy to see that the standard position of the polygon may not be unique, it depends on the choice of the longest edge, and of the choice of the special vertex that becomes the origin.

We can now present the classification of the polygons in standard position according to $m$ (the maximum number of integral points lying on the edges of the polygon), and also according to their number of edges, $n$. Obviously, the polygons ${\mathcal{P}}$ will have at most six edges, and at most five points on an edge.
% We will denote by $R$ the point $(0,1)$. Obviously, $R\in {\mathcal{P}}$

We leave the elementary details to the reader; we only remark that Pick's lemma is useful (for more details see \cite{Dum10}).

\begin{proposition}\label{Classification} Any polygon enclosing six lattice points is equivalent to exactly one from the following list
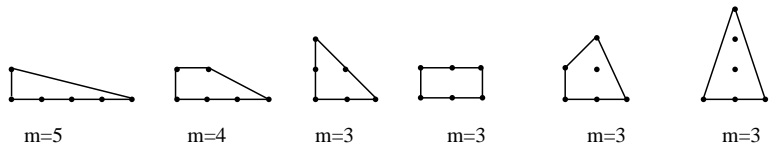
\begin{figure}[H]
\begin{center}
{
\begin{pspicture}(0,-0.941)(10.47825,0.911)
\psdots[dotsize=0.074](0.037,-0.336)
\psdots[dotsize=0.074](0.437,-0.336)
\psdots[dotsize=0.074](0.837,-0.336)
\psdots[dotsize=0.074](1.237,-0.336)
\psdots[dotsize=0.074](1.637,-0.336)
\psline[linewidth=0.02cm](0.017,-0.336)(1.617,-0.336)
\psdots[dotsize=0.074](0.037,0.064)
\psline[linewidth=0.02cm](0.037,0.064)(0.037,-0.296)
\psline[linewidth=0.02cm](0.017,0.084)(1.617,-0.316)
\psdots[dotsize=0.074](2.237,-0.336)
\psdots[dotsize=0.074](2.637,-0.336)
\psdots[dotsize=0.074](3.037,-0.336)
\psdots[dotsize=0.074](3.457,-0.336)
\psdots[dotsize=0.074](2.237,0.064)
\psdots[dotsize=0.074](2.657,0.064)
\psline[linewidth=0.02cm](2.217,0.084)(2.217,-0.336)
\psline[linewidth=0.02cm](2.217,-0.336)(3.457,-0.336)
\psline[linewidth=0.02cm](2.237,0.084)(2.637,0.084)
\psline[linewidth=0.02cm](2.657,0.084)(3.457,-0.336)
\psdots[dotsize=0.074](4.077,-0.336)
\psdots[dotsize=0.074](4.497,-0.336)
\psdots[dotsize=0.074](4.877,-0.336)
\psdots[dotsize=0.074](4.077,0.064)
\psdots[dotsize=0.074](4.077,0.464)
\psdots[dotsize=0.074](4.477,0.064)
\psline[linewidth=0.02cm](4.077,0.464)(4.077,-0.316)
\psline[linewidth=0.02cm](4.077,-0.336)(4.877,-0.336)
\psline[linewidth=0.02cm](4.077,0.484)(4.877,-0.316)
\usefont{T1}{ptm}{m}{n}
\rput(0.4616875,-0.801){\footnotesize m=5}
\usefont{T1}{ptm}{m}{n}
\rput(2.6346562,-0.801){\footnotesize m=4}
\psdots[dotsize=0.074](5.477,-0.316)
\psdots[dotsize=0.074](5.897,-0.316)
\psdots[dotsize=0.074](6.297,-0.316)
\psdots[dotsize=0.074](5.477,0.084)
\psdots[dotsize=0.074](5.897,0.084)
\psdots[dotsize=0.074](6.277,0.084)
\psline[linewidth=0.02cm](5.477,0.084)(5.477,-0.296)
\psline[linewidth=0.02cm](5.477,0.084)(6.277,0.084)
\psline[linewidth=0.02cm](6.297,0.084)(6.297,-0.316)
\psline[linewidth=0.02cm](5.477,-0.316)(6.277,-0.316)
\psdots[dotsize=0.074](7.397,-0.336)
\psdots[dotsize=0.074](7.817,-0.336)
\psdots[dotsize=0.074](8.217,-0.336)
\psdots[dotsize=0.074](7.397,0.084)
\psdots[dotsize=0.074](7.817,0.064)
\psdots[dotsize=0.074](7.817,0.484)
\psline[linewidth=0.02cm](7.797,0.484)(7.397,0.084)
\psline[linewidth=0.02cm](7.397,0.084)(7.397,-0.316)
\psline[linewidth=0.02cm](7.377,-0.336)(8.197,-0.336)
\psline[linewidth=0.02cm](7.817,0.504)(8.197,-0.316)
\usefont{T1}{ptm}{m}{n}
\rput(4.331844,-0.801){\footnotesize  m=3}
\usefont{T1}{ptm}{m}{n}
\rput(6.0849686,-0.801){\footnotesize  m=3}
\usefont{T1}{ptm}{m}{n}
\rput(7.9491873,-0.801){\footnotesize  m=3}
\psdots[dotsize=0.074](9.237,-0.336)
\psdots[dotsize=0.074](9.657,-0.336)
\psdots[dotsize=0.074](10.057,-0.336)
\psdots[dotsize=0.074](9.657,0.064)
\psdots[dotsize=0.074](9.657,0.464)
\psdots[dotsize=0.074](9.657,0.864)
\psline[linewidth=0.02cm](9.637,0.884)(9.237,-0.316)
\psline[linewidth=0.02cm](9.657,0.864)(10.057,-0.336)
\psline[linewidth=0.02cm](9.237,-0.336)(10.057,-0.336)
\usefont{T1}{ptm}{m}{n}
\rput(9.741844,-0.801){\footnotesize  m=3}
\end{pspicture} 
}
\end{center}
\end{figure}
% Generated with LaTeXDraw 2.0.2
% Mon Feb 01 00:06:17 MST 2010
% \usepackage[usenames,dvipsnames]{pstricks}
% \usepackage{epsfig}
% \usepackage{pst-grad} % For gradients
% \usepackage{pst-plot} % For axes
%\scalebox{1} % Change this value to rescale the drawing.
% Generated with LaTeXDraw 2.0.2
% Mon Feb 01 00:08:52 MST 2010
% \usepackage[usenames,dvipsnames]{pstricks}
% \usepackage{epsfig}
% \usepackage{pst-grad} % For gradients
% \usepackage{pst-plot} % For axes
% Generated with LaTeXDraw 2.0.2
% Mon Feb 01 00:10:12 MST 2010
% \usepackage[usenames,dvipsnames]{pstricks}
% \usepackage{epsfig}
% \usepackage{pst-grad} % For gradients
% \usepackage{pst-plot} % For axes
% Generated with LaTeXDraw 2.0.2
% Mon Feb 01 13:48:34 MST 2010
% \usepackage[usenames,dvipsnames]{pstricks}
% \usepackage{epsfig}
% \usepackage{pst-grad} % For gradients
% \usepackage{pst-plot} % For axes
\begin{figure}[H]
\begin{center}
\scalebox{1} % Change this value to rescale the drawing.
{
\begin{pspicture}(0,-1.8639235)(10.626249,1.8639235)
\psdots[dotsize=0.074](8.4375,-1.210486)
\psdots[dotsize=0.074](8.8475,-1.210486)
\psdots[dotsize=0.074](8.8475,-0.82048607)
\psdots[dotsize=0.074](8.8475,-0.42048606)
\psdots[dotsize=0.074](8.8475,0.0)
\psdots[dotsize=0.074](9.257501,0.0)
\psline[linewidth=0.02cm](8.837501,0.019513939)(8.4375,-1.2004861)
\psline[linewidth=0.02cm](9.257501,0.0)(8.8575,-1.2004861)
\psline[linewidth=0.02cm](8.8575,0.0)(9.257501,0.0)
\psline[linewidth=0.02cm](8.4375,-1.2004861)(8.837501,-1.2004861)
\psdots[dotsize=0.074](9.617499,-1.220486)
\psdots[dotsize=0.074](10.037499,-1.220486)
\psdots[dotsize=0.074](9.617499,-0.8004861)
\psdots[dotsize=0.074](10.037499,-0.8004861)
\psdots[dotsize=0.074](10.037499,-0.38048607)
\psdots[dotsize=0.074](10.4575,-0.8004861)
\psline[linewidth=0.02cm](9.617499,-0.78048605)(9.617499,-1.220486)
\psline[linewidth=0.02cm](9.5975,-1.220486)(10.0175,-1.220486)
\psline[linewidth=0.02cm](10.0175,-0.38048607)(9.617499,-0.78048605)
\psline[linewidth=0.02cm](10.037499,-0.34048605)(10.4775,-0.8004861)
\psline[linewidth=0.02cm](10.4575,-0.8004861)(10.037499,-1.220486)
\usefont{T1}{ptm}{m}{n}
\rput(0.5825,-1.6754861){\footnotesize m=2, n=3}
\usefont{T1}{ptm}{m}{n}
\rput(2.0125,-1.6754861){\footnotesize m=2, n=3}
\usefont{T1}{ptm}{m}{n}
\rput(3.52875,-1.6754861){\footnotesize m=2, n=3}
\usefont{T1}{ptm}{m}{n}
\rput(5.315,-1.6754861){\footnotesize m=2, n=4}
\usefont{T1}{ptm}{m}{n}
\rput(7.030312,-1.6754861){\footnotesize m=2, n=4}
\usefont{T1}{ptm}{m}{n}
\rput(8.580312,-1.6754861){\footnotesize m=2, n=4}
\usefont{T1}{ptm}{m}{n}
\rput(9.995312,-1.6754861){\footnotesize m=2, n=5}
\psdots[dotsize=0.074](0.4339062,-1.2093749)
\psdots[dotsize=0.074](0.8339062,-1.229375)
\psdots[dotsize=0.074](0.8339062,-0.8093749)
\psdots[dotsize=0.074](0.8339062,-0.38937494)
\psdots[dotsize=0.074](0.8339062,-0.02937495)
\psdots[dotsize=0.074](1.2539062,1.590625)
\psline[linewidth=0.02cm](0.4139062,-1.2093749)(0.8139062,-1.2093749)
\psline[linewidth=0.02cm](0.8339062,-1.229375)(1.2539062,1.5706251)
\psline[linewidth=0.02cm](0.4139062,-1.1893749)(1.2339061,1.5706251)
\psdots[dotsize=0.074](1.6339062,-1.2093749)
\psdots[dotsize=0.074](2.0339062,-1.2093749)
\psdots[dotsize=0.074](2.0339062,-0.8093749)
\psdots[dotsize=0.074](2.0539062,-0.40937495)
\psdots[dotsize=0.074](2.8339062,1.590625)
\psline[linewidth=0.02cm](1.6139061,-1.2093749)(2.0339062,-1.2093749)
\psline[linewidth=0.02cm](2.0339062,-1.2093749)(2.8339062,1.590625)
\psline[linewidth=0.02cm](1.6139061,-1.1893749)(2.8139062,1.590625)
\psdots[dotsize=0.074](2.4339063,0.41062504)
\psdots[dotsize=0.074](3.2339063,-1.2093749)
\psdots[dotsize=0.074](3.6339061,-1.2093749)
\psdots[dotsize=0.074](3.6339061,-0.8093749)
\psdots[dotsize=0.074](3.2339063,-0.82937497)
\psdots[dotsize=0.074](3.6339061,-0.40937495)
\psdots[dotsize=0.074](4.033906,0.37062505)
\psline[linewidth=0.02cm](3.2339063,-0.8093749)(3.2339063,-1.2093749)
\psline[linewidth=0.02cm](3.2339063,-1.2093749)(3.6139061,-1.2093749)
\psline[linewidth=0.02cm](3.2339063,-0.8093749)(4.013906,0.37062505)
\psline[linewidth=0.02cm](4.033906,0.37062505)(3.653906,-1.2093749)
\psdots[dotsize=0.074](4.853906,-1.229375)
\psdots[dotsize=0.074](5.2339063,-1.229375)
\psdots[dotsize=0.074](5.2339063,-0.84937495)
\psdots[dotsize=0.074](5.2339063,-0.42937493)
\psdots[dotsize=0.074](5.2339063,-0.02937495)
\psdots[dotsize=0.074](5.6539063,-0.84937495)
\psdots[dotsize=0.074](6.6539063,-1.229375)
\psdots[dotsize=0.074](7.033906,-1.229375)
\psdots[dotsize=0.074](7.053906,-0.82937497)
\psdots[dotsize=0.074](7.053906,-0.40937495)
\psdots[dotsize=0.074](7.033906,-0.02937495)
\psdots[dotsize=0.074](7.453906,-0.42937493)
\psline[linewidth=0.02cm](4.853906,-1.229375)(5.2139063,-1.229375)
\psline[linewidth=0.02cm](5.2339063,-1.229375)(5.6539063,-0.84937495)
\psline[linewidth=0.02cm](5.2339063,-0.00937495)(5.6539063,-0.82937497)
\psline[linewidth=0.02cm](4.833906,-1.1893749)(5.2139063,-0.00937495)
\psline[linewidth=0.02cm](6.6339064,-1.229375)(7.013906,-1.229375)
\psline[linewidth=0.02cm](7.033906,-1.2093749)(7.433906,-0.42937493)
\psline[linewidth=0.02cm](7.033906,-0.00937495)(7.433906,-0.40937495)
\psline[linewidth=0.02cm](7.033906,-0.00937495)(6.6339064,-1.2093749)
\usefont{T1}{ptm}{m}{n}
\rput(0.93015623,1.6967361){\footnotesize (2,7)}
\usefont{T1}{ptm}{m}{n}
\rput(2.5301561,1.6967361){\footnotesize (3,7)}
\usefont{T1}{ptm}{m}{n}
\rput(2.8101563,0.4367361){\footnotesize (2,5)}
\usefont{T1}{ptm}{m}{n}
\rput(4.3901563,0.4167361){\footnotesize (2,5)}
\end{pspicture} 
}
\end{center}
\caption{}
\end{figure}

\end{proposition}
We now recall that any rational convex polygon ${\mathcal{P}}$ in ${\mathbb{R}^{n}}$ enclosing a fixed number of integer lattice points defines an $n$ dimensional projective toric variety $X_{\mathcal{P}}$ endowed with an ample line bundle on $X_{\mathcal{P}}$ which has the integer points of the polygon as sections. We get the following result

\begin{corollary} Any toric surface endowed with an ample line bundle with six sections is completely described by exactly one of the polygons from the above list.

\end{corollary}

\section{Triple Point Analysis.}

%In our case, we obtain that any plane convex polygon enclosing six points corresponds to the datum of toric varieties and a line bundle on them with six sections. 
We first observe that six, the number of integer points enclosed by the polygon, represents exactly the number of conditions imposed by a triple point. We will now classify all polygons from Proposition \ref{Classification} for which their corresponding linear system becomes empty when imposing a triple point. There are two methods for testing the emptiness of these linear systems: an algebraic method and a geometric method and we will briefly describe them below.  For the algebraic approach, checking that a linear system is non-empty when imposing a triple point reduces to showing that the conditions imposed by a triple point in ${\mathbb{P}^{2}}$ are dependent. For this, one needs to look at the rank of the six by six matrix where the first column represents the sections of the line bundle and the other five columns represent all first and second derivatives in $x$ and $y$. We conclude that the six  conditions are dependent if and only if the determinant of the matrix is identically zero.
% and therefore the linear system with a triple point becomes non-empty. 
The geometric method for testing when a planar linear system is empty is to explicitly find it and show that it contains no curve, using ${\mathbb{P}^{2}}$ as a minimal model for the surface $X$ and writing its resolution of singularities.

\begin{remark}\label{Elimination} The corresponding linear systems of the following polygons are non-empty when imposing a base point with multiplicity three.
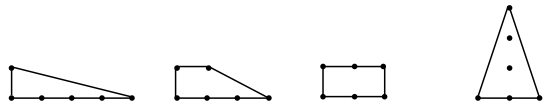
\begin{figure}[H]
\begin{center}
\scalebox{1} % Change this value to rescale the drawing.
{
\begin{pspicture}(0,-0.941)(7.47825,0.921)
\psdots[dotsize=0.074](0.037,-0.336)
\psdots[dotsize=0.074](0.437,-0.336)
\psdots[dotsize=0.074](0.837,-0.336)
\psdots[dotsize=0.074](1.237,-0.336)
\psdots[dotsize=0.074](1.637,-0.336)
\psline[linewidth=0.02cm](0.017,-0.336)(1.617,-0.336)
\psdots[dotsize=0.074](0.037,0.064)
\psline[linewidth=0.02cm](0.037,0.064)(0.037,-0.296)
\psline[linewidth=0.02cm](0.017,0.084)(1.617,-0.316)
\psdots[dotsize=0.074](2.237,-0.336)
\psdots[dotsize=0.074](2.637,-0.336)
\psdots[dotsize=0.074](3.037,-0.336)
\psdots[dotsize=0.074](3.457,-0.336)
\psdots[dotsize=0.074](2.237,0.064)
\psdots[dotsize=0.074](2.657,0.064)
\psline[linewidth=0.02cm](2.217,0.084)(2.217,-0.336)
\psline[linewidth=0.02cm](2.217,-0.336)(3.457,-0.336)
\psline[linewidth=0.02cm](2.237,0.084)(2.637,0.084)
\psline[linewidth=0.02cm](2.657,0.084)(3.457,-0.336)
\usefont{T1}{ptm}{m}{n}
\rput(0.4616875,-0.801){\footnotesize}
\usefont{T1}{ptm}{m}{n}
\rput(2.6346562,-0.801){\footnotesize}
\psdots[dotsize=0.074](4.177,-0.316)
\psdots[dotsize=0.074](4.597,-0.316)
\psdots[dotsize=0.074](4.997,-0.316)
\psdots[dotsize=0.074](4.177,0.084)
\psdots[dotsize=0.074](4.597,0.084)
\psdots[dotsize=0.074](4.977,0.084)
\psline[linewidth=0.02cm](4.177,0.084)(4.177,-0.296)
\psline[linewidth=0.02cm](4.177,0.084)(4.977,0.084)
\psline[linewidth=0.02cm](4.997,0.084)(4.997,-0.316)
\psline[linewidth=0.02cm](4.177,-0.316)(4.977,-0.316)
\usefont{T1}{ptm}{m}{n}
\rput(4.784969,-0.801){\footnotesize}
\psdots[dotsize=0.074](6.237,-0.336)
\psdots[dotsize=0.074](6.657,-0.336)
\psdots[dotsize=0.074](7.057,-0.336)
\psdots[dotsize=0.074](6.657,0.064)
\psdots[dotsize=0.074](6.657,0.464)
\psdots[dotsize=0.074](6.657,0.864)
\psline[linewidth=0.02cm](6.637,0.884)(6.237,-0.316)
\psline[linewidth=0.02cm](6.657,0.864)(7.057,-0.336)
\psline[linewidth=0.02cm](6.237,-0.336)(7.057,-0.336)
\usefont{T1}{ptm}{m}{n}
\rput(6.7418437,-0.801){\footnotesize}
\end{pspicture} 
}
\end{center}
\end{figure}

% Generated with LaTeXDraw 2.0.2
% Mon Feb 01 00:20:09 MST 2010
% \usepackage[usenames,dvipsnames]{pstricks}
% \usepackage{epsfig}
% \usepackage{pst-grad} % For gradients
% \usepackage{pst-plot} % For axes
% Generated with LaTeXDraw 2.0.2
% Mon Feb 01 00:22:41 MST 2010
% \usepackage[usenames,dvipsnames]{pstricks}
% \usepackage{epsfig}
% \usepackage{pst-grad} % For gradients
% \usepackage{pst-plot} % For axes
% Generated with LaTeXDraw 2.0.2
% Mon Feb 01 00:45:20 MST 2010
% \usepackage[usenames,dvipsnames]{pstricks}
% \usepackage{epsfig}
% \usepackage{pst-grad} % For gradients
% \usepackage{pst-plot} % For axes
%\scalebox{1} % Change this value to rescale the drawing.
% Generated with LaTeXDraw 2.0.2
% Mon Feb 01 13:43:22 MST 2010
% \usepackage[usenames,dvipsnames]{pstricks}
% \usepackage{epsfig}
% \usepackage{pst-grad} % For gradients
% \usepackage{pst-plot} % For axes
\begin{figure}[H]
\begin{center}
\scalebox{1} % Change this value to rescale the drawing.
{
\begin{pspicture}(0,-1.7515556)(6.4605937,1.7315556)
\psdots[dotsize=0.074](5.5835943,-1.1065556)
\psdots[dotsize=0.074](5.993594,-1.1065556)
\psdots[dotsize=0.074](5.993594,-0.71655554)
\psdots[dotsize=0.074](5.993594,-0.31655556)
\psdots[dotsize=0.074](5.993594,0.10344444)
\psdots[dotsize=0.074](6.4035935,0.10344444)
\psline[linewidth=0.02cm](5.983593,0.123444445)(5.5835943,-1.0965556)
\psline[linewidth=0.02cm](6.4035935,0.10344444)(6.0035944,-1.0965556)
\psline[linewidth=0.02cm](6.0035944,0.10344444)(6.4035935,0.10344444)
\psline[linewidth=0.02cm](5.5835943,-1.0965556)(5.983593,-1.0965556)
\usefont{T1}{ptm}{m}{n}
\rput(0.5232813,-1.5915556){\footnotesize}
\usefont{T1}{ptm}{m}{n}
\rput(2.1879687,-1.6115556){\footnotesize}
\usefont{T1}{ptm}{m}{n}
\rput(3.923281,-1.6115556){\footnotesize}
\usefont{T1}{ptm}{m}{n}
\rput(5.6732807,-1.5915556){\footnotesize}
\psdots[dotsize=0.074](0.34,-1.1054444)
\psdots[dotsize=0.074](0.72,-1.1054444)
\psdots[dotsize=0.074](0.72,-0.6854445)
\psdots[dotsize=0.074](0.72,-0.28544444)
\psdots[dotsize=0.074](0.72,0.11455555)
\psdots[dotsize=0.074](1.14,1.6745555)
\psline[linewidth=0.02cm](0.32,-1.1054444)(0.72,-1.1054444)
\psline[linewidth=0.02cm](0.72,-1.0854445)(1.14,1.6545556)
\psline[linewidth=0.02cm](0.32,-1.0854445)(1.12,1.6945555)
\psdots[dotsize=0.074](1.74,-1.1054444)
\psdots[dotsize=0.074](2.14,-1.1054444)
\psdots[dotsize=0.074](2.14,-0.70544446)
\psdots[dotsize=0.074](2.14,-0.32544443)
\psdots[dotsize=0.074](2.16,0.09455556)
\psdots[dotsize=0.074](2.56,-0.70544446)
\psline[linewidth=0.02cm](1.74,-1.1054444)(2.14,0.09455556)
\psline[linewidth=0.02cm](1.74,-1.1054444)(2.12,-1.1054444)
\psline[linewidth=0.02cm](2.14,-1.1054444)(2.56,-0.70544446)
\psline[linewidth=0.02cm](2.16,0.11455555)(2.56,-0.70544446)
\psdots[dotsize=0.074](3.94,-1.1054444)
\psdots[dotsize=0.074](3.94,-0.30544445)
\psdots[dotsize=0.074](3.94,0.09455556)
\psdots[dotsize=0.074](3.94,-0.70544446)
\psdots[dotsize=0.074](3.54,-1.1054444)
\psdots[dotsize=0.074](4.34,-0.30544445)
\psline[linewidth=0.02cm](3.92,0.11455555)(3.54,-1.1054444)
\psline[linewidth=0.02cm](3.54,-1.1054444)(3.92,-1.1054444)
\psline[linewidth=0.02cm](3.94,-1.1054444)(4.34,-0.30544445)
\psline[linewidth=0.02cm](3.92,0.11455555)(4.32,-0.28544444)
\end{pspicture} 
}

\end{center}
\caption{}\end{figure}

\end{remark}
 
\begin{proof} It is easy to check that the algebraic conditions imposed by at least four sections on a line are always dependent. Indeed, we have two possible cases, if the line of sections is an edge, or if is enclosed by the polygon. For the first case, we can only have sections on two levels so the vanishing of the second derivative in $y$ gives a dependent condition (The same argument applies for case $m=3$ representing the embedded  ${\mathbb{P}^{1}}\times {\mathbb{P}^{1}}$). For the second case we notice that the vanishing of the first derivative in $y$ and the second derivative in $x$ and $y$ give two linearly dependent conditions.
\end{proof}

We will use the Remark \ref{Elimination} to eliminate the polygons that don't have the desired property and we obtain five polygons for which we will study the corresponding algebraic surfaces and linear systems using toric geometry methods.

For any polygon consider its fan by dualizing the polygon's angles and in the case that the toric variety obtained by gluing the cones is singular, take it's resolution of singularites. In this way we obtain the all the toric surfaces using ${\mathbb{P}^{2}}$ as a minimal model. The associated linear system may not have general points. In general, we will use the notations ${\mathcal{L}}_{d}([1,1]), {\mathcal{L}}_{d}([2,1]), {\mathcal{L}}_{d}([1,1,1])$ for linear systems of degree $d$ that pass through a base point with a defined tangent, a double point with a defined tangent or having a flex direction. For example, ${\mathcal{L}}_{4}([1,1,1]^3)$ represents quartics with three base points that are flex to the line joining any two of them. Since the base points are special, the linear systems will need a different analysis. 
We conclude the emptiness of each linear systems with a triple point by applying birational transformations and splitting off $-1$ curves. The last column of the table indicates the geometric conditions that corresponding to the infinitely near multiplicities. We obtain the following result:
\newpage

\begin{lemma}\label {Empty poly}
The linear systems corresponding to the following polygons become empty after imposing a triple point.
\end{lemma}

\begin{figure}[H]
\begin{center}
%15
\scalebox{.6} % Change this value to rescale the drawing.

{
\begin{pspicture}(0,-7.84)(8.809375,7.825)
\usefont{T1}{ptm}{m}{n}
\rput(0.08546875,6.31){1.}
\pspolygon[linewidth=0.01](1.204375,6.82)(0.604375,6.22)(0.604375,5.62)(1.804375,5.62)
\psdots[dotsize=0.08](0.604375,6.22)
\psdots[dotsize=0.08](1.204375,6.82)
\psdots[dotsize=0.08](1.204375,6.22)
\psdots[dotsize=0.08](1.804375,5.62)
\psdots[dotsize=0.08](1.204375,5.62)
\psdots[dotsize=0.08](0.604375,5.62)
\usefont{T1}{ptm}{m}{n}
\rput(0.12265625,4.33){2.}
\psdots[dotsize=0.08](0.404375,3.62)
\psdots[dotsize=0.08](1.604375,5.42)
\psdots[dotsize=0.08](1.004375,3.62)
\psdots[dotsize=0.08](0.404375,3.02)
\psdots[dotsize=0.08](1.004375,4.22)
\psdots[dotsize=0.08](1.004375,3.02)
\usefont{T1}{ptm}{m}{n}
\rput(0.11921875,2.33){3.}
\psdots[dotsize=0.08](1.004375,2.82)
\psdots[dotsize=0.08](0.404375,2.22)
\psdots[dotsize=0.08](1.004375,1.62)
\psdots[dotsize=0.08](1.004375,2.22)
\psdots[dotsize=0.08](1.604375,2.22)
\psdots[dotsize=0.08](0.404375,1.62)
\usefont{T1}{ptm}{m}{n}
\rput(0.12375,0.73){4.}
\pspolygon[linewidth=0.01](1.004375,1.22)(1.004375,0.22)(2.004375,0.22)
\psdots[dotsize=0.08](1.004375,1.22)
\psdots[dotsize=0.08](1.004375,0.72)
\psdots[dotsize=0.08](1.504375,0.72)
\psdots[dotsize=0.08](1.004375,0.22)
\psdots[dotsize=0.08](2.004375,0.22)
\psdots[dotsize=0.08](1.484375,0.22)
\usefont{T1}{ptm}{m}{n}
\rput(0.09796875,-1.13){5.}
\psline[linewidth=0.01cm](2.404375,7.02)(2.404375,-2.98)
\psline[linewidth=0.01cm](6.804375,7.02)(6.804375,-2.98)
\psline[linewidth=0.01cm](0.004375,7.82)(8.804375,7.82)
\psline[linewidth=0.01cm](0.004375,7.02)(8.804375,7.02)
\usefont{T1}{ptm}{m}{n}
\rput(3.9832811,7.515){\footnotesize linear system of the toric variety}
\usefont{T1}{ptm}{m}{n}
\rput(1.3315625,7.515){\footnotesize polygon}
\pspolygon[linewidth=0.01](7.724375,6.44)(6.924375,5.44)(8.524375,5.44)
\pspolygon[linewidth=0.01](7.724375,4.84)(6.924375,3.84)(8.524375,3.84)
\pspolygon[linewidth=0.01](7.724375,2.64)(6.924375,1.64)(8.524375,1.64)
\usefont{T1}{ptm}{m}{n}
\rput(4.505781,5.97){${\mathcal{L}}_{3} ([1,1,1],1)$}
\psline[linewidth=0.01cm](7.604375,5.8034177)(7.840154,5.820072)
\psline[linewidth=0.01cm](7.6380577,5.9033456)(7.8064713,5.7201443)
\psline[linewidth=0.01cm](7.78963,5.92)(7.654899,5.70349)
\pscustom[linewidth=0.01]
{
\newpath
\moveto(6.864375,4.06)
\lineto(6.884375,4.03)
\curveto(6.894375,4.015)(6.914375,3.98)(6.924375,3.96)
\curveto(6.934375,3.94)(6.944375,3.905)(6.944375,3.89)
\curveto(6.944375,3.875)(6.934375,3.84)(6.924375,3.82)
\curveto(6.914375,3.8)(6.899375,3.76)(6.894375,3.74)
\curveto(6.889375,3.72)(6.894375,3.68)(6.904375,3.66)
\curveto(6.914375,3.64)(6.934375,3.61)(6.964375,3.58)
}
\pscustom[linewidth=0.01]
{
\newpath
\moveto(7.664375,6.66)
\lineto(7.684375,6.63)
\curveto(7.694375,6.615)(7.714375,6.58)(7.724375,6.56)
\curveto(7.734375,6.54)(7.744375,6.505)(7.744375,6.49)
\curveto(7.744375,6.475)(7.734375,6.44)(7.724375,6.42)
\curveto(7.714375,6.4)(7.699375,6.36)(7.694375,6.34)
\curveto(7.689375,6.32)(7.694375,6.28)(7.704375,6.26)
\curveto(7.714375,6.24)(7.734375,6.21)(7.764375,6.18)
}
\pscircle[linewidth=0.01,dimen=outer](8.524375,5.44){0.12}
\psline[linewidth=0.01cm](7.604375,4.1634173)(7.840154,4.1800723)
\psline[linewidth=0.01cm](7.6380577,4.2633452)(7.8064713,4.0801444)
\psline[linewidth=0.01cm](7.78963,4.28)(7.654899,4.06349)
\psline[linewidth=0.01cm](7.604375,1.9434175)(7.840154,1.9600722)
\psline[linewidth=0.01cm](7.6380577,2.0433455)(7.8064713,1.8601444)
\psline[linewidth=0.01cm](7.78963,2.06)(7.654899,1.8434898)
\pscustom[linewidth=0.01]
{
\newpath
\moveto(8.404375,3.88)
\lineto(8.453787,3.86)
\curveto(8.478493,3.85)(8.523787,3.84)(8.544375,3.84)
\curveto(8.564963,3.84)(8.610257,3.85)(8.634963,3.86)
\curveto(8.659669,3.87)(8.680258,3.9)(8.67614,3.92)
\curveto(8.672023,3.94)(8.651434,3.965)(8.634963,3.97)
\curveto(8.618492,3.975)(8.581434,3.965)(8.560846,3.95)
\curveto(8.540257,3.935)(8.51967,3.895)(8.51967,3.87)
\curveto(8.51967,3.845)(8.523787,3.8)(8.5279045,3.78)
\curveto(8.5320215,3.76)(8.5361395,3.735)(8.5361395,3.72)
}
\pscustom[linewidth=0.01]
{
\newpath
\moveto(7.624375,4.92)
\lineto(7.664375,4.9)
\curveto(7.684375,4.89)(7.714375,4.86)(7.724375,4.84)
\curveto(7.734375,4.82)(7.739375,4.78)(7.734375,4.76)
\curveto(7.729375,4.74)(7.714375,4.71)(7.684375,4.68)
}
\usefont{T1}{ptm}{m}{n}
\rput(4.545781,4.27){${\mathcal{L}}_{4} ([1,1,1],[1,1],[2,1])$}
\usefont{T1}{ptm}{m}{n}
\rput(4.525781,2.03){${\mathcal{L}}_{3} ([1,1],[1,1])$}
\usefont{T1}{ptm}{m}{n}
\rput(4.4057813,0.51){${\mathcal{L}}_{2} $}
\usefont{T1}{ptm}{m}{n}
\rput(4.6753125,-1.63){ ${\mathcal{L}}_{4}([1,1,1]^3).$}
\pspolygon[linewidth=0.01](2.004375,-4.38)(1.404375,-4.38)(1.404375,-3.78)(2.804375,-3.38)(2.804375,-3.78)
\pspolygon[linewidth=0.01](1.404375,-4.78)(1.804375,-6.18)(3.204375,-6.18)(2.604375,-5.38)
\psline[linewidth=0.01cm](0.004375,-2.98)(8.804375,-2.98)
\pspolygon[linewidth=0.01](1.604375,-6.98)(2.204375,-6.58)(2.804375,-6.98)(3.204375,-7.78)(2.804375,-7.78)
\usefont{T1}{ptm}{m}{n}
\rput(0.3478125,-3.285){\footnotesize Note:}
\pspolygon[linewidth=0.01](4.204375,-3.78)(4.804375,-3.18)(5.404375,-3.78)(5.404375,-4.38)(4.804375,-4.38)
\pspolygon[linewidth=0.01](5.204375,-4.78)(4.204375,-6.18)(6.004375,-6.18)(6.004375,-5.58)
\pspolygon[linewidth=0.01](4.204375,-6.98)(4.604375,-7.78)(5.204375,-7.78)(5.204375,-6.38)
\psdots[dotsize=0.08](2.604375,-6.18)
\psdots[dotsize=0.08](2.804375,-3.38)
\psdots[dotsize=0.08](2.004375,-3.78)
\psdots[dotsize=0.08](1.404375,-3.78)
\psdots[dotsize=0.08](2.804375,-3.78)
\psdots[dotsize=0.08](2.004375,-4.38)
\psdots[dotsize=0.08](1.404375,-4.38)
\psdots[dotsize=0.08](4.804375,-3.78)
\psdots[dotsize=0.08](5.404375,-4.38)
\psdots[dotsize=0.08](5.404375,-3.78)
\psdots[dotsize=0.08](4.804375,-3.18)
\psdots[dotsize=0.08](4.204375,-3.78)
\psdots[dotsize=0.08](4.804375,-4.38)
\psdots[dotsize=0.08](1.804375,-5.38)
\psdots[dotsize=0.08](2.604375,-5.38)
\psdots[dotsize=0.08](1.804375,-6.18)
\psdots[dotsize=0.08](1.404375,-4.78)
\psdots[dotsize=0.08](3.204375,-6.18)
\psdots[dotsize=0.08](5.204375,-5.58)
\psdots[dotsize=0.08](5.204375,-6.18)
\psdots[dotsize=0.08](6.004375,-6.18)
\psdots[dotsize=0.08](4.204375,-6.18)
\psdots[dotsize=0.08](5.204375,-4.78)
\psdots[dotsize=0.08](6.004375,-5.58)
\psdots[dotsize=0.08](5.204375,-7.78)
\psdots[dotsize=0.08](2.804375,-7.78)
\psdots[dotsize=0.08](3.204375,-7.78)
\psdots[dotsize=0.08](2.804375,-6.98)
\psdots[dotsize=0.08](2.204375,-6.58)
\psdots[dotsize=0.08](1.604375,-6.98)
\psdots[dotsize=0.08](4.604375,-7.78)
\psdots[dotsize=0.08](5.204375,-6.98)
\psdots[dotsize=0.08](4.604375,-6.98)
\psdots[dotsize=0.08](5.204375,-6.38)
\psdots[dotsize=0.08](4.204375,-6.98)
\psdots[dotsize=0.08](2.204375,-6.98)
\usefont{T1}{ptm}{m}{n}
\rput(3.5557814,-3.87){$\cong$}
\usefont{T1}{ptm}{m}{n}
\rput(3.5557814,-5.47){$\cong$}
\usefont{T1}{ptm}{m}{n}
\rput(3.5557814,-7.27){$\cong$}
\psdots[dotsize=0.08](7.184375,-1.98)
\psdots[dotsize=0.08](8.384375,-1.98)
\psdots[dotsize=0.08](7.784375,-1.18)
\psline[linewidth=0.01cm](7.784375,-1.18)(7.164375,-1.98)
\psline[linewidth=0.01cm](7.164375,-1.98)(8.364375,-1.98)
\psline[linewidth=0.01cm](7.784375,-1.16)(8.364375,-1.96)
\pscustom[linewidth=0.01]
{
\newpath
\moveto(6.944375,-1.9)
\lineto(6.984375,-1.92)
\curveto(7.004375,-1.93)(7.039375,-1.95)(7.054375,-1.96)
\curveto(7.069375,-1.97)(7.099375,-1.98)(7.114375,-1.98)
\curveto(7.129375,-1.98)(7.169375,-1.98)(7.194375,-1.98)
\curveto(7.219375,-1.98)(7.264375,-1.985)(7.284375,-1.99)
\curveto(7.304375,-1.995)(7.334375,-2.015)(7.344375,-2.03)
\curveto(7.354375,-2.045)(7.364375,-2.07)(7.364375,-2.1)
}
\pscustom[linewidth=0.01]
{
\newpath
\moveto(8.304375,-1.78)
\lineto(8.314375,-1.83)
\curveto(8.319375,-1.855)(8.329375,-1.9)(8.334375,-1.92)
\curveto(8.339375,-1.94)(8.354375,-1.96)(8.384375,-1.96)
}
\pscustom[linewidth=0.01]
{
\newpath
\moveto(8.384375,-1.98)
\lineto(8.414375,-2.01)
\curveto(8.429375,-2.025)(8.449375,-2.06)(8.454375,-2.08)
\curveto(8.459375,-2.1)(8.459375,-2.13)(8.444375,-2.16)
}
\pscustom[linewidth=0.01]
{
\newpath
\moveto(7.584375,-1.24)
\lineto(7.634375,-1.24)
\curveto(7.659375,-1.24)(7.699375,-1.23)(7.714375,-1.22)
\curveto(7.729375,-1.21)(7.749375,-1.195)(7.764375,-1.18)
}
\pscustom[linewidth=0.01]
{
\newpath
\moveto(7.784375,-1.16)
\lineto(7.804375,-1.13)
\curveto(7.814375,-1.115)(7.834375,-1.085)(7.844375,-1.07)
\curveto(7.854375,-1.055)(7.879375,-1.035)(7.924375,-1.02)
}
\psdots[dotsize=0.08](1.324375,-1.36)
\psdots[dotsize=0.08](0.984375,-1.98)
\psdots[dotsize=0.08](0.984375,-2.38)
\psdots[dotsize=0.08](0.184375,-2.8)
\psdots[dotsize=0.08](1.784375,0.02)
\psdots[dotsize=0.08](1.004375,-2.78)
\pspolygon[linewidth=0.01](0.404375,2.22)(0.404375,1.62)(1.004375,1.62)(1.604375,2.22)(1.004375,2.82)
\pspolygon[linewidth=0.01](0.204375,-2.78)(1.004375,-2.78)(1.804375,0.02)
\rput{45.0}(4.057047,-4.8345776){\psarc[linewidth=0.01](7.864375,2.48){0.2}{-0.0}{180.0}}
\rput{-180.0}(16.92875,3.72){\psarc[linewidth=0.01](8.464375,1.86){0.2}{-0.0}{180.0}}
\pspolygon[linewidth=0.01](0.404375,3.62)(0.404375,3.02)(1.004375,3.02)(1.604375,5.42)
\end{pspicture} 
}

\end{center}
\caption{}\end{figure}
All the linear systems from the table become empty when imposing a triple point.
For example the fourth polygon in the above table describes a projective plane embedded by a linear system of conics, ${\mathcal{L}}_{2}$. By imposing a triple point ${\mathcal{L}}_{2}$ becomes empty. One can obtain more polygons with an empty linear system by rotating or by shifting the main ones by any integer numbers.

\section{Triple points in ${\mathbb{P}}^ 2$.}

We denote by $V_d$ the image of the Veronese embedding
$v_d:\mathbb{P}^2 \to \mathbb{P}^{d(d+3)/2}$
that transforms the plane curves of degree $d$ to hyperplane sections of the Veronese variety $V_d$.
We degenerate $V_d$ into a union of disjoint surfaces and ordinary planes and we place one point on each one of the disjoint surfaces. 
The surfaces are chosen such that the restriction of a hyperplane section to each one of them to be linear system that becomes empty when we impose a triple point. 
We conclude that any hyperplane section to $V_d$ needs to contain all disjoint surfaces, and in particular all of the coordinate points of the ambient projective space covered in this way. Therefore if $V_{d}$ degenerates exactly into a union of disjoint special surfaces and planes (or quadrics) with no points left over we conclude that the desired linear system is empty, and therefore it has the expected dimension. Using semicontinuity this argument can easily be extended to any degeneration as in \cite{CDM07}.
In order to give an inductive proof for triple points in the projective plane we will first analyze triple points in ${\mathbb{P}}^ 1\times{\mathbb{P}}^ 1$.

%\section{Triple Points in ${\mathbb{P}}^ 1\times{\mathbb{P}}^ 1$.}

 We will only prove the most difficult case when the linear systems in ${\mathbb{P}}^ 1\times{\mathbb{P}}^ 1$ with virtual dimension $-1$ are empty.
The general case will follow by induction, but it was already proved in a similar way using algebraic methods by T. Lenarcik in \cite{Len08}.

\begin{lemma}\label{bidegree $(5,n)$}
Fix $n\geq 3$. Then linear systems of bidegree $(5,n)$ for $n\neq 4, (11,n), (2, 4n-9), (8,2n-3)$ and an arbitrary number of triple points have the expected dimension.
\end{lemma}

\begin{proof} 
\begin{itemize}
\item For any linear systems of bidegree $(5,n)$ we find a skew $n+1$ set of surfaces and we place each of the $n+1$ triple points in one of the surfaces. We denote the degenerations presented below as $C_{5}^{5}$, $C_{5}^{6}$, $C_{5}^{8}$ and $C_{5}^{3}$
\begin{figure}[H]
\begin{center}
\scalebox{1}
{
\begin{pspicture}(0,-1.165625)(10.2,1.205625)
\psframe[linewidth=0.01,dimen=outer](2.19,0.884375)(0.18,-1.125625)
\psline[linewidth=0.01cm](0.595,0.869375)(0.595,-1.130625)
\psline[linewidth=0.01cm](0.995,0.869375)(0.995,-1.130625)
\psline[linewidth=0.01cm](1.395,0.869375)(1.395,-1.130625)
\psline[linewidth=0.01cm](1.795,0.869375)(1.795,-1.130625)
\psline[linewidth=0.01cm](0.195,0.469375)(2.195,0.469375)
\psline[linewidth=0.01cm](0.195,0.069375)(2.195,0.069375)
\psline[linewidth=0.01cm](0.195,-0.330625)(2.195,-0.330625)
\psline[linewidth=0.01cm](0.195,-0.730625)(2.195,-0.730625)
\pspolygon[linewidth=0.0104,fillstyle=vlines,hatchwidth=0.01,hatchangle=45.0,hatchsep=0.04](0.195,0.869375)(0.195,0.069375)(0.995,0.869375)
\pspolygon[linewidth=0.01,fillstyle=vlines,hatchwidth=0.01,hatchangle=45.0,hatchsep=0.04](1.395,0.869375)(2.195,0.869375)(2.195,0.069375)
\pspolygon[linewidth=0.01,fillstyle=vlines,hatchwidth=0.01,hatchangle=45.0,hatchsep=0.04](1.395,0.469375)(1.795,0.069375)(1.795,-0.330625)(0.995,0.069375)(0.995,0.469375)
\pspolygon[linewidth=0.01,fillstyle=vlines,hatchwidth=0.01,hatchangle=45.0,hatchsep=0.04](0.595,0.069375)(1.395,-0.330625)(1.395,-0.730625)(0.995,-0.730625)(0.595,-0.330625)
\pspolygon[linewidth=0.01,fillstyle=vlines,hatchwidth=0.01,hatchangle=45.0,hatchsep=0.04](0.195,-0.330625)(0.995,-1.130625)(0.195,-1.130625)
\pspolygon[linewidth=0.01,fillstyle=vlines,hatchwidth=0.01,hatchangle=45.0,hatchsep=0.04](2.195,-0.330625)(2.195,-1.130625)(1.395,-1.130625)
\usefont{T1}{ptm}{m}{n}
\rput(0.35671875,1.044375){\footnotesize 1}
\usefont{T1}{ptm}{m}{n}
\rput(0.785,1.044375){\footnotesize 2}
\usefont{T1}{ptm}{m}{n}
\rput(1.1751562,1.044375){\footnotesize 3}
\usefont{T1}{ptm}{m}{n}
\rput(1.5870312,1.044375){\footnotesize 4}
\usefont{T1}{ptm}{m}{n}
\rput(1.976875,1.044375){\footnotesize 5}
\usefont{T1}{ptm}{m}{n}
\rput(0.03671875,0.664375){\footnotesize 1}
\usefont{T1}{ptm}{m}{n}
\rput(0.065,0.264375){\footnotesize 2}
\usefont{T1}{ptm}{m}{n}
\rput(0.05515625,-0.135625){\footnotesize 3}
\usefont{T1}{ptm}{m}{n}
\rput(0.06703125,-0.535625){\footnotesize 4}
\usefont{T1}{ptm}{m}{n}
\rput(0.056875,-0.935625){\footnotesize 5}
\psline[linewidth=0.01cm](2.995,0.869375)(2.995,-1.130625)
\psline[linewidth=0.01cm](3.395,0.869375)(3.395,-1.130625)
\psline[linewidth=0.01cm](3.795,0.869375)(3.795,-1.130625)
\psline[linewidth=0.01cm](4.195,0.869375)(4.195,-1.130625)
\psline[linewidth=0.01cm](2.595,0.469375)(4.995,0.469375)
\psline[linewidth=0.01cm](2.595,0.069375)(4.995,0.069375)
\psline[linewidth=0.01cm](2.595,-0.330625)(4.995,-0.330625)
\psline[linewidth=0.01cm](2.595,-0.730625)(4.995,-0.730625)
\pspolygon[linewidth=0.0104,fillstyle=vlines,hatchwidth=0.01,hatchangle=45.0,hatchsep=0.04](2.595,0.869375)(2.595,0.069375)(3.395,0.869375)
\pspolygon[linewidth=0.01,fillstyle=vlines,hatchwidth=0.01,hatchangle=45.0,hatchsep=0.04](4.195,0.869375)(4.995,0.869375)(4.995,0.069375)
\pspolygon[linewidth=0.01,fillstyle=vlines,hatchwidth=0.01,hatchangle=45.0,hatchsep=0.04](2.595,-0.330625)(3.395,-1.130625)(2.595,-1.130625)
\usefont{T1}{ptm}{m}{n}
\rput(2.7567186,1.044375){\footnotesize 1}
\usefont{T1}{ptm}{m}{n}
\rput(3.185,1.044375){\footnotesize 2}
\usefont{T1}{ptm}{m}{n}
\rput(3.5751562,1.044375){\footnotesize 3}
\usefont{T1}{ptm}{m}{n}
\rput(3.9870312,1.044375){\footnotesize 4}
\usefont{T1}{ptm}{m}{n}
\rput(4.376875,1.044375){\footnotesize 5}
\usefont{T1}{ptm}{m}{n}
\rput(2.4367187,0.664375){\footnotesize 1}
\usefont{T1}{ptm}{m}{n}
\rput(2.465,0.264375){\footnotesize 2}
\usefont{T1}{ptm}{m}{n}
\rput(2.4551563,-0.135625){\footnotesize 3}
\usefont{T1}{ptm}{m}{n}
\rput(2.4670312,-0.535625){\footnotesize 4}
\usefont{T1}{ptm}{m}{n}
\rput(2.456875,-0.935625){\footnotesize 5}
\psframe[linewidth=0.01,dimen=outer](5.015,0.869375)(2.595,-1.130625)
\psline[linewidth=0.01cm](4.595,0.869375)(4.595,-1.130625)
\pspolygon[linewidth=0.01,fillstyle=vlines,hatchwidth=0.01,hatchangle=45.0,hatchsep=0.04](3.795,0.869375)(4.595,0.069375)(4.595,-0.330625)(4.195,-0.330625)
\pspolygon[linewidth=0.01,fillstyle=vlines,hatchwidth=0.01,hatchangle=45.0,hatchsep=0.04](3.795,0.469375)(3.795,0.069375)(2.995,-0.330625)(2.995,0.069375)(3.395,0.469375)
\usefont{T1}{ptm}{m}{n}
\rput(4.782031,1.044375){\footnotesize 6}
\psline[linewidth=0.01cm](5.795,0.869375)(5.795,-1.130625)
\psline[linewidth=0.01cm](6.195,0.869375)(6.195,-1.130625)
\psline[linewidth=0.01cm](6.595,0.869375)(6.595,-1.130625)
\psline[linewidth=0.01cm](6.995,0.869375)(6.995,-1.130625)
\psline[linewidth=0.01cm](5.395,0.469375)(8.595,0.469375)
\psline[linewidth=0.01cm](5.395,0.069375)(8.595,0.069375)
\psline[linewidth=0.01cm](5.395,-0.330625)(8.595,-0.330625)
\psline[linewidth=0.01cm](5.395,-0.730625)(8.595,-0.730625)
\pspolygon[linewidth=0.0104,fillstyle=vlines,hatchwidth=0.01,hatchangle=45.0,hatchsep=0.04](5.395,0.869375)(5.395,0.069375)(6.195,0.869375)
\pspolygon[linewidth=0.01,fillstyle=vlines,hatchwidth=0.01,hatchangle=45.0,hatchsep=0.04](7.795,0.869375)(8.595,0.869375)(8.595,0.069375)
\pspolygon[linewidth=0.01,fillstyle=vlines,hatchwidth=0.01,hatchangle=45.0,hatchsep=0.04](5.395,-0.330625)(6.195,-1.130625)(5.395,-1.130625)
\usefont{T1}{ptm}{m}{n}
\rput(5.556719,1.044375){\footnotesize 1}
\usefont{T1}{ptm}{m}{n}
\rput(5.985,1.044375){\footnotesize 2}
\usefont{T1}{ptm}{m}{n}
\rput(6.3751564,1.044375){\footnotesize 3}
\usefont{T1}{ptm}{m}{n}
\rput(6.787031,1.044375){\footnotesize 4}
\usefont{T1}{ptm}{m}{n}
\rput(7.176875,1.044375){\footnotesize 5}
\usefont{T1}{ptm}{m}{n}
\rput(5.2367187,0.664375){\footnotesize 1}
\usefont{T1}{ptm}{m}{n}
\rput(5.265,0.264375){\footnotesize 2}
\usefont{T1}{ptm}{m}{n}
\rput(5.255156,-0.135625){\footnotesize 3}
\usefont{T1}{ptm}{m}{n}
\rput(5.267031,-0.535625){\footnotesize 4}
\usefont{T1}{ptm}{m}{n}
\rput(5.256875,-0.935625){\footnotesize 5}
\psframe[linewidth=0.01,dimen=outer](8.615,0.869375)(5.395,-1.130625)
\psline[linewidth=0.01cm](7.395,0.869375)(7.395,-1.130625)
\usefont{T1}{ptm}{m}{n}
\rput(7.5820312,1.044375){\footnotesize 6}
\psline[linewidth=0.01cm](7.795,0.869375)(7.795,-1.130625)
\psline[linewidth=0.01cm](8.195,0.869375)(8.195,-1.130625)
\pspolygon[linewidth=0.01,fillstyle=vlines,hatchwidth=0.01,hatchangle=45.0,hatchsep=0.04](5.795,0.069375)(6.595,0.869375)(6.595,0.069375)
\pspolygon[linewidth=0.01,fillstyle=vlines,hatchwidth=0.01,hatchangle=45.0,hatchsep=0.04](6.995,0.869375)(7.395,0.869375)(7.795,0.469375)(7.395,0.069375)(6.995,0.469375)
\pspolygon[linewidth=0.01,fillstyle=vlines,hatchwidth=0.01,hatchangle=45.0,hatchsep=0.04](5.795,-0.330625)(6.995,0.069375)(6.595,-0.730625)(6.195,-0.730625)
\pspolygon[linewidth=0.01,fillstyle=vlines,hatchwidth=0.01,hatchangle=45.0,hatchsep=0.04](6.595,-1.130625)(6.995,-0.330625)(7.395,-0.330625)(7.395,-0.730625)(6.995,-1.130625)
\pspolygon[linewidth=0.01,fillstyle=vlines,hatchwidth=0.01,hatchangle=45.0,hatchsep=0.04](7.795,0.069375)(8.195,0.069375)(8.195,-0.330625)(7.395,-1.130625)
\pspolygon[linewidth=0.01,fillstyle=vlines,hatchwidth=0.01,hatchangle=45.0,hatchsep=0.04](7.795,-1.130625)(8.595,-0.330625)(8.595,-1.130625)
\psline[linewidth=0.01cm](9.395,0.869375)(9.395,-1.130625)
\psline[linewidth=0.01cm](8.995,0.469375)(10.195,0.469375)
\psline[linewidth=0.01cm](8.995,0.069375)(10.195,0.069375)
\psline[linewidth=0.01cm](8.995,-0.330625)(10.195,-0.330625)
\psline[linewidth=0.01cm](8.995,-0.730625)(10.195,-0.730625)
\pspolygon[linewidth=0.0104,fillstyle=vlines,hatchwidth=0.01,hatchangle=45.0,hatchsep=0.04](8.995,0.869375)(8.995,0.069375)(9.795,0.869375)
\usefont{T1}{ptm}{m}{n}
\rput(9.156719,1.044375){\footnotesize 1}
\usefont{T1}{ptm}{m}{n}
\rput(9.585,1.044375){\footnotesize 2}
\usefont{T1}{ptm}{m}{n}
\rput(8.836719,0.664375){\footnotesize 1}
\usefont{T1}{ptm}{m}{n}
\rput(8.865,0.264375){\footnotesize 2}
\usefont{T1}{ptm}{m}{n}
\rput(8.855156,-0.135625){\footnotesize 3}
\usefont{T1}{ptm}{m}{n}
\rput(8.867031,-0.535625){\footnotesize 4}
\usefont{T1}{ptm}{m}{n}
\rput(8.856875,-0.935625){\footnotesize 5}
\psframe[linewidth=0.01,dimen=outer](10.195,0.869375)(8.995,-1.130625)
\psline[linewidth=0.01cm](9.795,0.869375)(9.795,-1.130625)
\pspolygon[linewidth=0.01,fillstyle=vlines,hatchwidth=0.01,hatchangle=45.0,hatchsep=0.04](10.195,0.869375)(10.195,0.069375)(9.395,0.069375)
\pspolygon[linewidth=0.01,fillstyle=vlines,hatchwidth=0.01,hatchangle=45.0,hatchsep=0.04](8.995,-0.330625)(9.795,-0.330625)(8.995,-1.130625)
\pspolygon[linewidth=0.01,fillstyle=vlines,hatchwidth=0.01,hatchangle=45.0,hatchsep=0.04](9.395,-1.130625)(10.195,-0.330625)(10.195,-1.130625)
\usefont{T1}{ptm}{m}{n}
\rput(9.975156,1.044375){\footnotesize 3}
\usefont{T1}{ptm}{m}{n}
\rput(7.9820313,1.044375){\footnotesize 7}
\usefont{T1}{ptm}{m}{n}
\rput(8.382031,1.044375){\footnotesize 8}
\psline[linewidth=0.01cm](1.195,0.869375)(0.195,-0.130625)
\psline[linewidth=0.01cm](0.195,-0.130625)(1.195,-1.130625)
\psline[linewidth=0.01cm](1.195,-1.130625)(2.195,-0.130625)
\psline[linewidth=0.01cm](1.195,0.869375)(2.195,-0.130625)
\psline[linewidth=0.01cm](0.615,0.289375)(1.775,-0.550625)
\psline[linewidth=0.01cm](3.595,0.869375)(2.595,-0.130625)
\psline[linewidth=0.01cm](2.595,-0.130625)(3.595,-1.130625)
\pspolygon[linewidth=0.01,fillstyle=vlines,hatchwidth=0.01,hatchangle=45.0,hatchsep=0.04](4.995,-0.330625)(4.195,-1.130625)(4.995,-1.130625)
\pspolygon[linewidth=0.01,fillstyle=vlines,hatchwidth=0.01,hatchangle=46.0,hatchsep=0.04](3.395,-0.330625)(3.395,-0.730625)(3.795,-1.130625)(4.195,-0.730625)(3.795,-0.330625)
\psline[linewidth=0.01cm](3.995,-1.130625)(4.995,-0.130625)
\psline[linewidth=0.01cm](4.995,-0.130625)(3.995,0.869375)
\psline[linewidth=0.01cm](3.735,0.869375)(4.315,-0.830625)
\psline[linewidth=0.01cm](2.915,-0.450625)(4.035,-0.030625)
\psline[linewidth=0.01cm](6.395,0.869375)(5.395,-0.130625)
\psline[linewidth=0.01cm](5.395,-0.130625)(6.395,-1.130625)
\psline[linewidth=0.01cm](7.595,-1.130625)(8.595,-0.130625)
\psline[linewidth=0.01cm](8.595,-0.130625)(7.595,0.869375)
\psline[linewidth=0.01cm](7.795,0.669375)(7.795,0.069375)
\psline[linewidth=0.01cm](7.875,0.589375)(7.335,-1.130625)
\psline[linewidth=0.01cm](6.655,0.869375)(7.555,-0.390625)
\psline[linewidth=0.01cm](7.115,0.229375)(5.615,-0.370625)
\psline[linewidth=0.01cm](7.155,0.189375)(6.475,-1.130625)
\end{pspicture} 
}

\end{center}
\caption{}\end{figure}
For every $n>2$, $n\neq 4$ take $i\in \{3,5,6,8 \}$ such that $\frac{n-i}{4}$ is an integer, $k$. For any arbitrary $n$ we consider the degeneration
$C_{5}^{n}=C_{5}^{i}+k C_{5}^{3}$ where the sum of two blocks means attaching the two disjoint blocks together along the edge of length $5$.

\item For linear systems of bidegree $(11,n)$ and 2$n+2$ triple points we find skew surfaces. We denote the degenerations presented below by $C_{11}^{2}$, $C_{11}^{3}$, and $C_{11}^{4}.$ For every $n>2$ take $i\in \{2,3,4\}$ such that $\frac{n-i}{3}$ is an integer, $k$. For any arbitrary $n$ we consider the degeneration
$C_{11}^{n}=C_{11}^{i}+k C_{11}^{2}.$

\item For curves of bidegree $(2,4n-9)$ we consider the degeneration $C_{2}^{4n-9}$ given by $(n-3)C_{2}^{3}$ (in particular, $C_{2}^{11}=3C_{2}^{3}$) and for $C_{8}^{2n-3}$ we use combinations of $C_{8}^{3}$ and $C_{8}^{5}$

\begin{figure}[H]
\begin{center}
\scalebox{1} % Change this value to rescale the drawing.
{
\begin{pspicture}(0,-1.453125)(9.398125,1.488125)
\psline[linewidth=0.01cm](5.393125,0.151875)(5.393125,-1.448125)
\psline[linewidth=0.01cm](5.793125,0.151875)(5.793125,-1.448125)
\psline[linewidth=0.01cm](6.193125,0.151875)(6.193125,-1.448125)
\psline[linewidth=0.01cm](6.593125,0.151875)(6.593125,-1.448125)
\psline[linewidth=0.01cm](4.993125,-0.248125)(9.393125,-0.248125)
\psline[linewidth=0.01cm](4.993125,-0.648125)(9.393125,-0.648125)
\psline[linewidth=0.01cm](4.993125,-1.048125)(9.393125,-1.048125)
\pspolygon[linewidth=0.0104,fillstyle=vlines,hatchwidth=0.01,hatchangle=45.0,hatchsep=0.04](4.993125,0.151875)(4.993125,-0.648125)(5.793125,0.151875)
\usefont{T1}{ptm}{m}{n}
\rput(5.154844,0.326875){\footnotesize 1}
\usefont{T1}{ptm}{m}{n}
\rput(5.583125,0.326875){\footnotesize 2}
\usefont{T1}{ptm}{m}{n}
\rput(5.9732814,0.326875){\footnotesize 3}
\usefont{T1}{ptm}{m}{n}
\rput(6.385156,0.326875){\footnotesize 4}
\usefont{T1}{ptm}{m}{n}
\rput(6.775,0.326875){\footnotesize 5}
\usefont{T1}{ptm}{m}{n}
\rput(4.8348436,-0.053125){\footnotesize 1}
\usefont{T1}{ptm}{m}{n}
\rput(4.863125,-0.453125){\footnotesize 2}
\usefont{T1}{ptm}{m}{n}
\rput(4.853281,-0.853125){\footnotesize 3}
\usefont{T1}{ptm}{m}{n}
\rput(4.865156,-1.253125){\footnotesize 4}
\psframe[linewidth=0.01,dimen=outer](9.393125,0.151875)(4.993125,-1.448125)
\psline[linewidth=0.01cm](6.993125,0.151875)(6.993125,-1.448125)
\usefont{T1}{ptm}{m}{n}
\rput(7.180156,0.326875){\footnotesize 6}
\psline[linewidth=0.01cm](7.393125,0.151875)(7.393125,-1.448125)
\psline[linewidth=0.01cm](7.793125,0.151875)(7.793125,-1.448125)
\psline[linewidth=0.01cm](8.193125,0.151875)(8.193125,-1.448125)
\psline[linewidth=0.01cm](8.593125,0.151875)(8.593125,-1.448125)
\usefont{T1}{ptm}{m}{n}
\rput(7.5801563,0.326875){\footnotesize 7}
\usefont{T1}{ptm}{m}{n}
\rput(7.974219,0.326875){\footnotesize 8}
\usefont{T1}{ptm}{m}{n}
\rput(8.379844,0.326875){\footnotesize 9}
\usefont{T1}{ptm}{m}{n}
\rput(8.76875,0.326875){\footnotesize 10}
\psline[linewidth=0.01cm](8.993125,0.151875)(8.993125,-1.448125)
\usefont{T1}{ptm}{m}{n}
\rput(9.154843,0.326875){\footnotesize 11}
\pspolygon[linewidth=0.01,fillstyle=vlines,hatchwidth=0.01,hatchangle=45.0,hatchsep=0.04](4.993125,-1.048125)(5.393125,-0.648125)(5.793125,-1.448125)(4.993125,-1.448125)
\pspolygon[linewidth=0.01,fillstyle=vlines,hatchwidth=0.01,hatchangle=45.0,hatchsep=0.04](6.193125,0.151875)(5.793125,-0.248125)(6.193125,-0.648125)(6.593125,-0.248125)(6.593125,0.151875)
\pspolygon[linewidth=0.01,fillstyle=vlines,hatchwidth=0.01,hatchangle=45.0,hatchsep=0.04](5.793125,-0.648125)(6.593125,-1.048125)(6.593125,-1.448125)(6.193125,-1.448125)(5.793125,-1.048125)
\pspolygon[linewidth=0.01,fillstyle=vlines,hatchwidth=0.01,hatchangle=45.0,hatchsep=0.04](6.593125,-0.648125)(6.993125,-0.648125)(7.393125,-0.248125)(7.393125,0.151875)(6.993125,0.151875)
\pspolygon[linewidth=0.01,fillstyle=vlines,hatchwidth=0.01,hatchangle=45.0,hatchsep=0.04](6.993125,-1.048125)(7.393125,-0.648125)(7.793125,-0.648125)(7.393125,-1.448125)(6.993125,-1.448125)
\pspolygon[linewidth=0.01,fillstyle=vlines,hatchwidth=0.01,hatchangle=45.0,hatchsep=0.04](8.993125,-0.648125)(9.393125,-1.048125)(9.393125,-1.448125)(8.593125,-1.448125)
\pspolygon[linewidth=0.01,fillstyle=vlines,hatchwidth=0.01,hatchangle=45.0,hatchsep=0.04](8.593125,0.151875)(9.393125,0.151875)(9.393125,-0.648125)
\psline[linewidth=0.01cm](0.593125,1.151875)(0.593125,0.36012596)
\psline[linewidth=0.01cm](0.993125,1.151875)(0.993125,0.36012596)
\psline[linewidth=0.01cm](1.393125,1.151875)(1.393125,0.36012596)
\psline[linewidth=0.01cm](1.793125,1.151875)(1.793125,0.36012596)
\psline[linewidth=0.01cm](0.193125,0.751875)(4.593125,0.751875)
\psline[linewidth=0.01cm](0.193125,0.351875)(4.593125,0.351875)
\usefont{T1}{ptm}{m}{n}
\rput(0.35484374,1.326875){\footnotesize 1}
\usefont{T1}{ptm}{m}{n}
\rput(0.783125,1.326875){\footnotesize 2}
\usefont{T1}{ptm}{m}{n}
\rput(1.1732812,1.326875){\footnotesize 3}
\usefont{T1}{ptm}{m}{n}
\rput(1.5851562,1.326875){\footnotesize 4}
\usefont{T1}{ptm}{m}{n}
\rput(1.975,1.326875){\footnotesize 5}
\usefont{T1}{ptm}{m}{n}
\rput(0.03484375,0.946875){\footnotesize 1}
\usefont{T1}{ptm}{m}{n}
\rput(0.063125,0.546875){\footnotesize 2}
\psframe[linewidth=0.01,dimen=outer](4.593125,1.151875)(0.193125,0.351875)
\psline[linewidth=0.01cm](2.193125,1.151875)(2.193125,0.36012596)
\usefont{T1}{ptm}{m}{n}
\rput(2.3801563,1.326875){\footnotesize 6}
\psline[linewidth=0.01cm](2.593125,1.151875)(2.593125,0.36012596)
\psline[linewidth=0.01cm](2.993125,1.151875)(2.993125,0.36012596)
\psline[linewidth=0.01cm](3.393125,1.151875)(3.393125,0.36012596)
\psline[linewidth=0.01cm](3.793125,1.151875)(3.793125,0.36012596)
\usefont{T1}{ptm}{m}{n}
\rput(2.7801561,1.326875){\footnotesize 7}
\usefont{T1}{ptm}{m}{n}
\rput(3.1742187,1.326875){\footnotesize 8}
\usefont{T1}{ptm}{m}{n}
\rput(3.5798438,1.326875){\footnotesize 9}
\usefont{T1}{ptm}{m}{n}
\rput(3.96875,1.326875){\footnotesize 10}
\psline[linewidth=0.01cm](4.193125,1.151875)(4.193125,0.36012596)
\usefont{T1}{ptm}{m}{n}
\rput(4.3548436,1.326875){\footnotesize 11}
\psline[linewidth=0.01cm](0.593125,-0.248125)(0.593125,-1.4445562)
\psline[linewidth=0.01cm](0.993125,-0.248125)(0.993125,-1.4445562)
\psline[linewidth=0.01cm](1.393125,-0.248125)(1.393125,-1.4445562)
\psline[linewidth=0.01cm](1.793125,-0.248125)(1.793125,-1.4445562)
\psline[linewidth=0.01cm](0.193125,-0.648125)(4.593125,-0.648125)
\psline[linewidth=0.01cm](0.193125,-1.048125)(4.593125,-1.048125)
\usefont{T1}{ptm}{m}{n}
\rput(0.35484374,-0.073125){\footnotesize 1}
\usefont{T1}{ptm}{m}{n}
\rput(0.783125,-0.073125){\footnotesize 2}
\usefont{T1}{ptm}{m}{n}
\rput(1.1732812,-0.073125){\footnotesize 3}
\usefont{T1}{ptm}{m}{n}
\rput(1.5851562,-0.073125){\footnotesize 4}
\usefont{T1}{ptm}{m}{n}
\rput(1.975,-0.073125){\footnotesize 5}
\usefont{T1}{ptm}{m}{n}
\rput(0.03484375,-0.453125){\footnotesize 1}
\usefont{T1}{ptm}{m}{n}
\rput(0.063125,-0.853125){\footnotesize 2}
\usefont{T1}{ptm}{m}{n}
\rput(0.05328125,-1.253125){\footnotesize 3}
\psframe[linewidth=0.01,dimen=outer](4.593125,-0.248125)(0.193125,-1.4445562)
\psline[linewidth=0.01cm](2.193125,-0.248125)(2.193125,-1.4445562)
\usefont{T1}{ptm}{m}{n}
\rput(2.3801563,-0.073125){\footnotesize 6}
\psline[linewidth=0.01cm](2.593125,-0.248125)(2.593125,-1.4445562)
\psline[linewidth=0.01cm](2.993125,-0.248125)(2.993125,-1.4445562)
\psline[linewidth=0.01cm](3.393125,-0.248125)(3.393125,-1.4445562)
\psline[linewidth=0.01cm](3.793125,-0.248125)(3.793125,-1.4445562)
\usefont{T1}{ptm}{m}{n}
\rput(2.7801561,-0.073125){\footnotesize 7}
\usefont{T1}{ptm}{m}{n}
\rput(3.1742187,-0.073125){\footnotesize 8}
\usefont{T1}{ptm}{m}{n}
\rput(3.5798438,-0.073125){\footnotesize 9}
\usefont{T1}{ptm}{m}{n}
\rput(3.96875,-0.073125){\footnotesize 10}
\psline[linewidth=0.01cm](4.193125,-0.248125)(4.193125,-1.4445562)
\usefont{T1}{ptm}{m}{n}
\rput(4.3548436,-0.073125){\footnotesize 11}
\pspolygon[linewidth=0.01,fillstyle=vlines,hatchwidth=0.01,hatchangle=45.0,hatchsep=0.04](0.193125,1.151875)(0.993125,0.351875)(0.193125,0.351875)
\pspolygon[linewidth=0.01,fillstyle=vlines,hatchwidth=0.01,hatchangle=45.0,hatchsep=0.04](1.793125,1.151875)(2.593125,0.351875)(1.793125,0.351875)
\pspolygon[linewidth=0.01,fillstyle=vlines,hatchwidth=0.01,hatchangle=45.0,hatchsep=0.04](3.393125,1.151875)(4.193125,0.351875)(3.393125,0.351875)
\pspolygon[linewidth=0.01,fillstyle=vlines,hatchwidth=0.01,hatchangle=45.0,hatchsep=0.04](0.193125,-0.648125)(0.993125,-1.448125)(0.193125,-1.448125)
\pspolygon[linewidth=0.01,fillstyle=vlines,hatchwidth=0.01,hatchangle=45.0,hatchsep=0.04](1.393125,-0.648125)(2.193125,-1.448125)(1.393125,-1.448125)
\pspolygon[linewidth=0.01,fillstyle=vlines,hatchwidth=0.01,hatchangle=45.0,hatchsep=0.04](2.593125,-0.648125)(3.393125,-1.448125)(2.593125,-1.448125)
\pspolygon[linewidth=0.01,fillstyle=vlines,hatchwidth=0.01,hatchangle=45.0,hatchsep=0.04](0.593125,1.151875)(1.393125,1.151875)(1.393125,0.351875)
\pspolygon[linewidth=0.01,fillstyle=vlines,hatchwidth=0.01,hatchangle=45.0,hatchsep=0.04](2.193125,1.151875)(2.993125,1.151875)(2.993125,0.351875)
\pspolygon[linewidth=0.01,fillstyle=vlines,hatchwidth=0.01,hatchangle=45.0,hatchsep=0.04](3.793125,1.151875)(4.593125,1.151875)(4.593125,0.351875)
\pspolygon[linewidth=0.01,fillstyle=vlines,hatchwidth=0.01,hatchangle=45.0,hatchsep=0.04](0.193125,-0.248125)(0.993125,-0.248125)(0.993125,-1.048125)
\pspolygon[linewidth=0.01,fillstyle=vlines,hatchwidth=0.01,hatchangle=45.0,hatchsep=0.04](1.393125,-0.248125)(2.193125,-0.248125)(2.193125,-1.048125)
\pspolygon[linewidth=0.01,fillstyle=vlines,hatchwidth=0.01,hatchangle=45.0,hatchsep=0.04](2.593125,-0.248125)(3.393125,-0.248125)(3.393125,-1.048125)
\pspolygon[linewidth=0.01,fillstyle=vlines,hatchwidth=0.01,hatchangle=45.0,hatchsep=0.04](3.793125,-0.248125)(4.593125,-0.248125)(4.593125,-1.048125)
\pspolygon[linewidth=0.01,fillstyle=vlines,hatchwidth=0.01,hatchangle=45.0,hatchsep=0.04](3.793125,-0.648125)(4.593125,-1.448125)(3.793125,-1.448125)
\psline[linewidth=0.01cm](5.993125,0.151875)(4.993125,-0.848125)
\psline[linewidth=0.01cm](5.433125,-0.408125)(5.893125,-1.428125)
\psline[linewidth=0.01cm](6.613125,-1.428125)(7.753125,0.151875)
\psline[linewidth=0.01cm](6.853125,-1.108125)(5.513125,-0.328125)
\psline[linewidth=0.01cm](6.353125,-0.808125)(6.893125,0.151875)
\psline[linewidth=0.01cm](8.393125,0.151875)(9.393125,-0.848125)
\psline[linewidth=0.01cm](8.493125,-1.428125)(8.933125,-0.388125)
\pspolygon[linewidth=0.01,fillstyle=vlines,hatchwidth=0.01,hatchangle=45.0,hatchsep=0.04](7.793125,0.151875)(7.793125,-0.248125)(8.593125,-0.648125)(8.593125,-0.248125)(8.193125,0.151875)
\pspolygon[linewidth=0.01,fillstyle=vlines,hatchwidth=0.01,hatchangle=45.0,hatchsep=0.04](7.793125,-1.448125)(7.793125,-1.048125)(8.193125,-0.648125)(8.593125,-1.048125)(8.193125,-1.448125)
\psline[linewidth=0.01cm](8.753125,-0.808125)(7.513125,-0.188125)
\psline[linewidth=0.01cm](7.913125,-0.408125)(7.593125,-1.428125)
\end{pspicture} 
}

\end{center}
\caption{}\end{figure}
\end{itemize}
\end{proof}

\begin{corollary}\label{main}
Linear systems in ${\mathbb{P}}^ {1}\times {\mathbb{P}}^ {1}$ with triple points of virtual dimension $-1$ are empty.
\end{corollary}

\begin{proof} We have to prove the statement for linear systems of bidegree $(6k-1, n)$ and $(3k-1, 2n-1).$
For the bidegree $(6k-1, n)$ we distinguish two cases. If $k$ is even $k=2k'$ we use the degeneration $C_{12k'-1}^{n}=k'C_{11}^{n}$; while if $k$ is odd of the form $2k'+1$ we use $C_{12k'+5}^{n}=C_{5}^{n}+k'C_{11}^{n}$, for $n\neq 4$. For $n=4$ we use the following degeneration for $C_{17}^{4}$ and we generalize this case by adding $C_{11}^{4}$ blocks
\begin{figure}[H]
\begin{center}
\scalebox{1}
{
\begin{pspicture}(0,-0.953125)(7.0,0.988125)
\psline[linewidth=0.01cm](0.595,0.651875)(0.595,-0.948125)
\psline[linewidth=0.01cm](0.995,0.651875)(0.995,-0.948125)
\psline[linewidth=0.01cm](1.395,0.651875)(1.395,-0.948125)
\psline[linewidth=0.01cm](1.795,0.651875)(1.795,-0.948125)
\psline[linewidth=0.01cm](0.195,0.251875)(6.995,0.251875)
\psline[linewidth=0.01cm](0.195,-0.148125)(6.995,-0.148125)
\psline[linewidth=0.01cm](0.195,-0.548125)(6.995,-0.548125)
\pspolygon[linewidth=0.0104,fillstyle=vlines,hatchwidth=0.01,hatchangle=45.0,hatchsep=0.04](0.195,0.651875)(0.195,-0.148125)(0.995,0.651875)
\usefont{T1}{ptm}{m}{n}
\rput(0.35671875,0.826875){\footnotesize 1}
\usefont{T1}{ptm}{m}{n}
\rput(0.785,0.826875){\footnotesize 2}
\usefont{T1}{ptm}{m}{n}
\rput(1.1751562,0.826875){\footnotesize 3}
\usefont{T1}{ptm}{m}{n}
\rput(1.5870312,0.826875){\footnotesize 4}
\usefont{T1}{ptm}{m}{n}
\rput(1.976875,0.826875){\footnotesize 5}
\usefont{T1}{ptm}{m}{n}
\rput(0.03671875,0.446875){\footnotesize 1}
\usefont{T1}{ptm}{m}{n}
\rput(0.065,0.046875){\footnotesize 2}
\usefont{T1}{ptm}{m}{n}
\rput(0.05515625,-0.353125){\footnotesize 3}
\usefont{T1}{ptm}{m}{n}
\rput(0.06703125,-0.753125){\footnotesize 4}
\psframe[linewidth=0.01,dimen=outer](6.995,0.651875)(0.195,-0.948125)
\psline[linewidth=0.01cm](2.195,0.651875)(2.195,-0.948125)
\usefont{T1}{ptm}{m}{n}
\rput(2.3820312,0.826875){\footnotesize 6}
\psline[linewidth=0.01cm](2.595,0.651875)(2.595,-0.948125)
\psline[linewidth=0.01cm](2.995,0.651875)(2.995,-0.948125)
\psline[linewidth=0.01cm](3.395,0.651875)(3.395,-0.948125)
\psline[linewidth=0.01cm](3.795,0.651875)(3.795,-0.948125)
\usefont{T1}{ptm}{m}{n}
\rput(2.7820313,0.826875){\footnotesize 7}
\usefont{T1}{ptm}{m}{n}
\rput(3.1760938,0.826875){\footnotesize 8}
\usefont{T1}{ptm}{m}{n}
\rput(3.5817187,0.826875){\footnotesize 9}
\usefont{T1}{ptm}{m}{n}
\rput(3.970625,0.826875){\footnotesize 10}
\psline[linewidth=0.01cm](4.195,0.651875)(4.195,-0.948125)
\usefont{T1}{ptm}{m}{n}
\rput(4.3567185,0.826875){\footnotesize 11}
\pspolygon[linewidth=0.01,fillstyle=vlines,hatchwidth=0.01,hatchangle=45.0,hatchsep=0.04](0.195,-0.548125)(0.595,-0.148125)(0.995,-0.948125)(0.195,-0.948125)
\psline[linewidth=0.01cm](5.795,0.651875)(5.795,-0.948125)
\psline[linewidth=0.01cm](5.395,0.651875)(5.395,-0.948125)
\psline[linewidth=0.01cm](4.995,0.651875)(4.995,-0.948125)
\psline[linewidth=0.01cm](4.595,0.651875)(4.595,-0.948125)
\psline[linewidth=0.01cm](6.595,0.651875)(6.595,-0.948125)
\psline[linewidth=0.01cm](6.195,0.651875)(6.195,-0.948125)
\usefont{T1}{ptm}{m}{n}
\rput(4.7696877,0.826875){\footnotesize 12}
\usefont{T1}{ptm}{m}{n}
\rput(5.1629686,0.826875){\footnotesize 13}
\usefont{T1}{ptm}{m}{n}
\rput(5.570781,0.826875){\footnotesize 14}
\usefont{T1}{ptm}{m}{n}
\rput(5.965781,0.826875){\footnotesize 15}
\usefont{T1}{ptm}{m}{n}
\rput(6.3701563,0.826875){\footnotesize 16}
\usefont{T1}{ptm}{m}{n}
\rput(6.769219,0.826875){\footnotesize 17}
\pspolygon[linewidth=0.01,fillstyle=vlines,hatchwidth=0.01,hatchangle=45.0,hatchsep=0.04](1.395,0.651875)(1.795,0.651875)(0.995,-0.548125)(0.995,0.251875)
\pspolygon[linewidth=0.01,fillstyle=vlines,hatchwidth=0.01,hatchangle=45.0,hatchsep=0.04](6.195,0.651875)(6.995,0.651875)(6.995,-0.148125)
\pspolygon[linewidth=0.01,fillstyle=vlines,hatchwidth=0.01,hatchangle=45.0,hatchsep=0.04](6.195,-0.948125)(6.595,-0.148125)(6.995,-0.548125)(6.995,-0.948125)
\psline[linewidth=0.01cm](1.195,0.651875)(0.195,-0.348125)
\psline[linewidth=0.01cm](5.995,0.651875)(6.995,-0.348125)
\psline[linewidth=0.01cm](0.635,0.091875)(1.135,-0.928125)
\psline[linewidth=0.01cm](0.975,-0.608125)(1.995,0.651875)
\pspolygon[linewidth=0.01,fillstyle=vlines,hatchwidth=0.01,hatchangle=45.0,hatchsep=0.04](1.395,-0.948125)(1.395,-0.148125)(2.195,-0.948125)
\pspolygon[linewidth=0.01,fillstyle=vlines,hatchwidth=0.01,hatchangle=45.0,hatchsep=0.04](5.795,-0.948125)(6.195,-0.548125)(5.795,-0.148125)(5.395,-0.548125)(5.395,-0.948125)
\pspolygon[linewidth=0.01,fillstyle=vlines,hatchwidth=0.01,hatchangle=45.0,hatchsep=0.04](5.795,0.651875)(6.195,0.251875)(6.195,-0.148125)(4.995,0.651875)
\pspolygon[linewidth=0.01,fillstyle=vlines,hatchwidth=0.01,hatchangle=45.0,hatchsep=0.04](5.395,0.251875)(5.395,-0.148125)(4.995,-0.548125)(4.595,0.251875)
\pspolygon[linewidth=0.01,fillstyle=vlines,hatchwidth=0.01,hatchangle=45.0,hatchsep=0.04](4.995,-0.948125)(4.595,-0.148125)(4.195,-0.148125)(4.195,-0.548125)(4.595,-0.948125)
\pspolygon[linewidth=0.01,fillstyle=vlines,hatchwidth=0.01,hatchangle=45.0,hatchsep=0.04](4.595,0.651875)(3.795,-0.148125)(3.795,0.651875)
\pspolygon[linewidth=0.01,fillstyle=vlines,hatchwidth=0.01,hatchangle=45.0,hatchsep=0.04](3.395,0.651875)(3.395,0.251875)(2.995,-0.148125)(2.595,0.251875)(2.995,0.651875)
\pspolygon[linewidth=0.01,fillstyle=vlines,hatchwidth=0.01,hatchangle=45.0,hatchsep=0.04](2.595,0.651875)(2.195,-0.148125)(1.795,-0.148125)(1.795,0.251875)(2.195,0.651875)
\pspolygon[linewidth=0.01,fillstyle=vlines,hatchwidth=0.01,hatchangle=45.0,hatchsep=0.04](2.595,-0.948125)(2.195,-0.548125)(2.595,-0.148125)(2.995,-0.548125)(2.995,-0.948125)
\pspolygon[linewidth=0.01,fillstyle=vlines,hatchwidth=0.01,hatchangle=45.0,hatchsep=0.04](3.395,-0.948125)(3.395,-0.148125)(4.195,-0.948125)
\psline[linewidth=0.01cm](2.375,-0.928125)(1.455,-0.028125)
\psline[linewidth=0.01cm](2.655,0.651875)(2.055,-0.628125)
\psline[linewidth=0.01cm](2.475,0.291875)(3.335,-0.928125)
\psline[linewidth=0.01cm](3.035,-0.508125)(3.715,0.651875)
\psline[linewidth=0.01cm](3.415,0.131875)(4.375,-0.928125)
\psline[linewidth=0.01cm](3.855,-0.348125)(4.795,0.631875)
\psline[linewidth=0.01cm](4.795,0.651875)(6.395,-0.308125)
\psline[linewidth=0.01cm](6.515,0.111875)(6.095,-0.948125)
\psline[linewidth=0.01cm](5.055,-0.948125)(4.475,0.291875)
\psline[linewidth=0.01cm](4.935,-0.708125)(5.755,0.071875)
\end{pspicture} 
}

\end{center}
\caption{}\end{figure}
For the bidegree $(3k-1, 2n-1)$ we reduce to the case when $k$ is odd of the form $2k'+1$ and depending on the parity of $k'$, if $6k'=6+12r$ we use the degeneration
$C_{8}^{2n-1}+rC_{11}^{2n-1}$ while $6k'=12r$ we use $C_{5}^{2n-1}+C_{8}^{2n-1}+(r-1)C_{11}^{2n-1}$.
\end{proof}

We can now obtain the following result

\begin{theorem}\label{triple plane}
$\mathcal{L}_d(3^n)$ has the expected dimension whenever $d \geq 5$.
\end{theorem}

\begin{proof}
It is enough to prove the theorem for the number of triple points for which the virtual dimension is $-1$ so in that case we claim that the linear system is empty.
Note that $\binom{d+2}{2}\equiv 0$ mod $6$ if $d\equiv \{2,7, 10, 11\}$ mod $6$; $\binom{d+2}{2}\equiv 1$ mod $6$ if $d\equiv \{0,9\}$ mod $6$;
$\binom{d+2}{2}\equiv 3$ mod $6$ if $d\equiv \{1, 4, 5, 8\}$ mod $6$ and $\binom{d+2}{2}\equiv 4$ mod $6$ if $d\equiv \{3, 6\}$ mod $6$.

We will use the induction step
$V_{12(k+1)+j}=V_{12k+j}+kC^{11}_{11}+C^{11}_{j+1}+V_{10}$
with $j=1,...,12$, $k\geq 0$, $(i,j)\neq (1,4)$ and to finish the proof we  present the degenerations of $V_{j}$ if $j\leq 12.$
\begin{figure}[H]
\begin{center}
\scalebox{1} % Change this value to rescale the drawing.
{
\begin{pspicture}(0,-1.33)(3.793125,1.33)
\usefont{T1}{ptm}{m}{n}
\rput(0.93671876,0.535){\footnotesize $V_{12k+j}$}
\psline[linewidth=0.02cm](0.403125,-0.1)(2.383125,-0.12)
\psline[linewidth=0.02cm](0.403125,-0.1)(0.403125,-1.3)
\psline[linewidth=0.02cm](0.403125,-1.3)(3.783125,-1.32)
\psline[linewidth=0.02cm](2.383125,-0.12)(3.743125,-1.3)
\psline[linewidth=0.02cm](2.383125,-0.12)(2.403125,-1.28)
\psline[linewidth=0.02cm](2.183125,-0.12)(2.203125,-1.3)
\usefont{T1}{ptm}{m}{n}
\rput(2.9067187,-0.845){\footnotesize $V_{10}$}
\usefont{T1}{ptm}{m}{n}
\rput(1.3967187,-0.805){\footnotesize $C^{11}_{12k+j+1}$}
\psline[linewidth=0.02cm](0.383125,1.32)(2.383125,-0.12)
\psline[linewidth=0.02cm](0.383125,0.1)(2.083125,0.1)
\psline[linewidth=0.02cm](0.383125,1.32)(0.403125,-0.14)
\end{pspicture} 
}

\end{center}
\caption{}\end{figure}

\begin{figure}[H]
\begin{center}
\scalebox{.65} % Change this value to rescale the drawing.
{
\begin{pspicture}(0,-6.705)(23.005,6.705)
\usefont{T1}{ptm}{m}{n}
\rput(5.2229686,6.075){\footnotesize k=0}
\psline[linewidth=0.01](4.8,5.9)(5.6,5.9)(5.6,6.3)
\usefont{T1}{ptm}{m}{n}
\rput(5.1535935,5.595){\footnotesize $j=1$}
\usefont{T1}{ptm}{m}{n}
\rput(5.1835938,5.195){\footnotesize $e=3$}
\usefont{T1}{ptm}{m}{n}
\rput(3.5135937,5.595){\footnotesize ${V}_{1}$}
\pscircle[linewidth=0.01,dimen=outer](2.22,5.86){0.1}
\pscircle[linewidth=0.01,dimen=outer](2.22,5.12){0.1}
\pscircle[linewidth=0.01,dimen=outer](3.18,5.12){0.1}
\psline[linewidth=0.01cm](0.0,4.3)(23.0,4.3)
\psline[linewidth=0.01cm](5.8,6.7)(5.8,-6.7)
\usefont{T1}{ptm}{m}{n}
\rput(9.953594,5.595){\footnotesize $j=2$}
\usefont{T1}{ptm}{m}{n}
\rput(9.983594,5.195){\footnotesize $e=0$}
\usefont{T1}{ptm}{m}{n}
\rput(8.513594,5.595){\footnotesize ${V}_{2}$}
\pspolygon[linewidth=0.01,fillstyle=vlines,hatchwidth=0.01,hatchangle=45.0,hatchsep=0.04](7.2,4.9)(7.2,6.1)(8.4,4.9)
\psline[linewidth=0.01cm](7.2,5.5)(7.8,5.5)
\psline[linewidth=0.01cm](7.8,5.5)(7.8,4.9)
\psline[linewidth=0.01cm](10.6,6.7)(10.6,-6.7)
\usefont{T1}{ptm}{m}{n}
\rput(15.353594,5.595){\footnotesize $j=3$}
\usefont{T1}{ptm}{m}{n}
\rput(15.383594,5.195){\footnotesize $e=4$}
\usefont{T1}{ptm}{m}{n}
\rput(13.7135935,5.595){\footnotesize ${V}_{3}$}
\pspolygon[linewidth=0.01](12.4,6.1)(12.4,4.9)(13.6,4.9)
\pscircle[linewidth=0.01,dimen=outer](12.42,4.9){0.1}
\pscircle[linewidth=0.01,dimen=outer](12.78,4.9){0.1}
\pscircle[linewidth=0.01,dimen=outer](13.18,4.9){0.1}
\pscircle[linewidth=0.01,dimen=outer](13.58,4.9){0.1}
\usefont{T1}{ptm}{m}{n}
\rput(22.353594,5.595){\footnotesize $j=5$}
\usefont{T1}{ptm}{m}{n}
\rput(22.383595,5.195){\footnotesize $e=3$}
\usefont{T1}{ptm}{m}{n}
\rput(19.913593,5.395){\footnotesize ${V}_{5}$}
\pspolygon[linewidth=0.01,fillstyle=vlines,hatchwidth=0.01,hatchangle=45.0,hatchsep=0.04](18.0,6.5)(18.0,5.7)(18.8,5.7)
\pspolygon[linewidth=0.01](18.0,5.7)(18.0,5.3)(19.2,5.3)(18.8,5.7)
\pspolygon[linewidth=0.01](19.2,5.3)(19.2,4.5)(20.0,4.5)
\pscircle[linewidth=0.01,dimen=outer](20.0,4.5){0.1}
\pscircle[linewidth=0.01,dimen=outer](19.58,4.88){0.1}
\pscircle[linewidth=0.01,dimen=outer](19.58,4.5){0.1}
\pspolygon[linewidth=0.01](1.8,3.5)(1.8,1.1)(4.2,1.1)
\usefont{T1}{ptm}{m}{n}
\rput(3.8535938,2.295){\footnotesize ${V}_{6}$}
\pspolygon[linewidth=0.01,fillstyle=vlines,hatchwidth=0.01,hatchangle=45.0,hatchsep=0.04](1.8,3.5)(1.8,2.7)(2.6,2.7)
\pspolygon[linewidth=0.01,fillstyle=vlines,hatchwidth=0.01,hatchangle=45.0,hatchsep=0.04](1.8,2.3)(2.6,2.3)(1.8,1.5)
\pspolygon[linewidth=0.01,fillstyle=vlines,hatchwidth=0.01,hatchangle=45.0,hatchsep=0.04](1.8,1.1)(2.6,1.9)(2.6,1.1)
\pspolygon[linewidth=0.01,fillstyle=vlines,hatchwidth=0.01,hatchangle=45.0,hatchsep=0.04](3.0,2.3)(3.0,1.5)(3.8,1.5)
\usefont{T1}{ptm}{m}{n}
\rput(5.233594,2.655){\footnotesize $j=6$}
\usefont{T1}{ptm}{m}{n}
\rput(5.2635937,2.255){\footnotesize $e=4$}
\pscircle[linewidth=0.01,dimen=outer](3.0,1.1){0.1}
\pscircle[linewidth=0.01,dimen=outer](3.4,1.1){0.1}
\pscircle[linewidth=0.01,dimen=outer](3.8,1.1){0.1}
\pscircle[linewidth=0.01,dimen=outer](4.2,1.12){0.1}
\psline[linewidth=0.01cm](0.0,0.1)(23.0,0.1)
\pspolygon[linewidth=0.01](6.8,3.5)(6.8,0.7)(9.6,0.7)
\pspolygon[linewidth=0.01,fillstyle=vlines,hatchwidth=0.01,hatchangle=45.0,hatchsep=0.04](6.8,3.5)(6.8,2.7)(7.6,2.7)
\pspolygon[linewidth=0.01,fillstyle=vlines,hatchwidth=0.01,hatchangle=45.0,hatchsep=0.04](6.8,2.3)(6.8,1.9)(7.2,1.5)(7.6,2.3)
\pspolygon[linewidth=0.01,fillstyle=vlines,hatchwidth=0.01,hatchangle=45.0,hatchsep=0.04](8.0,2.3)(7.6,1.9)(7.6,1.5)(8.0,1.5)(8.4,1.9)
\pspolygon[linewidth=0.01,fillstyle=vlines,hatchwidth=0.01,hatchangle=45.0,hatchsep=0.04](6.8,1.5)(6.8,0.7)(7.6,0.7)
\pspolygon[linewidth=0.01,fillstyle=vlines,hatchwidth=0.01,hatchangle=45.0,hatchsep=0.04](8.0,0.7)(8.4,0.7)(8.4,1.5)(7.6,1.1)
\pspolygon[linewidth=0.01,fillstyle=vlines,hatchwidth=0.01,hatchangle=45.0,hatchsep=0.04](8.8,1.5)(8.8,0.7)(9.6,0.7)
\psline[linewidth=0.01cm](6.8,1.1)(9.2,1.1)
\psline[linewidth=0.01cm](6.8,1.5)(8.8,1.5)
\psline[linewidth=0.01cm](6.8,1.9)(8.4,1.9)
\psline[linewidth=0.01cm](6.8,2.3)(8.0,2.3)
\psline[linewidth=0.01cm](7.2,3.1)(7.2,0.7)
\psline[linewidth=0.01cm](7.6,2.7)(7.6,0.7)
\psline[linewidth=0.01cm](8.0,2.3)(8.0,0.7)
\psline[linewidth=0.01cm](8.4,1.9)(8.4,0.7)
\psline[linewidth=0.01cm](9.2,1.1)(9.2,0.7)
\usefont{T1}{ptm}{m}{n}
\rput(10.033594,2.655){\footnotesize $j=7$}
\usefont{T1}{ptm}{m}{n}
\rput(10.063594,2.255){\footnotesize $e=0$}
\usefont{T1}{ptm}{m}{n}
\rput(9.153594,2.075){\footnotesize ${V}_{7}$}
\usefont{T1}{ptm}{m}{n}
\rput(15.433594,2.655){\footnotesize $j=8$}
\usefont{T1}{ptm}{m}{n}
\rput(15.4635935,2.255){\footnotesize $e=3$}
\usefont{T1}{ptm}{m}{n}
\rput(13.753593,2.275){\footnotesize ${V}_{8}$}
\pspolygon[linewidth=0.01](11.6,3.7)(11.6,0.5)(14.8,0.5)
\pspolygon[linewidth=0.01,fillstyle=vlines,hatchwidth=0.01,hatchangle=45.0,hatchsep=0.04](11.6,3.7)(11.6,2.9)(12.4,2.9)
\pspolygon[linewidth=0.01,fillstyle=vlines,hatchwidth=0.01,hatchangle=45.0,hatchsep=0.04](11.6,2.5)(11.6,1.7)(12.4,2.5)
\pspolygon[linewidth=0.01,fillstyle=vlines,hatchwidth=0.01,hatchangle=45.0,hatchsep=0.04](11.6,1.3)(11.6,0.5)(12.4,0.5)
\pspolygon[linewidth=0.01,fillstyle=vlines,hatchwidth=0.01,hatchangle=45.0,hatchsep=0.04](12.4,1.3)(13.2,1.7)(13.2,1.3)(12.8,0.9)(12.4,0.9)
\pspolygon[linewidth=0.01,fillstyle=vlines,hatchwidth=0.01,hatchangle=45.0,hatchsep=0.04](12.8,0.5)(13.6,0.5)(13.6,1.3)
\pspolygon[linewidth=0.01,fillstyle=vlines,hatchwidth=0.01,hatchangle=45.0,hatchsep=0.04](14.0,1.3)(14.0,0.5)(14.8,0.5)
\pscircle[linewidth=0.01,dimen=outer](12.02,1.7){0.1}
\pscircle[linewidth=0.01,dimen=outer](12.0,1.3){0.1}
\pscircle[linewidth=0.01,dimen=outer](13.56,1.72){0.1}
\psline[linewidth=0.01cm](12.0,3.3)(12.0,0.5)
\psline[linewidth=0.01cm](12.4,2.9)(12.4,0.5)
\psline[linewidth=0.01cm](12.8,2.5)(12.8,0.5)
\psline[linewidth=0.01cm](13.2,2.1)(13.2,0.5)
\psline[linewidth=0.01cm](13.6,1.7)(13.6,0.5)
\psline[linewidth=0.01cm](14.4,0.9)(14.4,0.5)
\psline[linewidth=0.01cm](11.6,0.9)(14.4,0.9)
\psline[linewidth=0.01cm](11.6,1.3)(14.0,1.3)
\psline[linewidth=0.01cm](11.6,1.7)(13.6,1.7)
\psline[linewidth=0.01cm](11.6,2.1)(13.2,2.1)
\psline[linewidth=0.01cm](11.6,2.5)(12.8,2.5)
\psline[linewidth=0.01cm](11.6,3.3)(12.0,3.3)
\pspolygon[linewidth=0.01](17.4,3.9)(17.4,0.3)(21.0,0.3)
\pspolygon[linewidth=0.01,fillstyle=vlines,hatchwidth=0.01,hatchangle=45.0,hatchsep=0.04](17.4,3.9)(17.4,3.1)(18.2,3.1)
\pspolygon[linewidth=0.01,fillstyle=vlines,hatchwidth=0.01,hatchangle=45.0,hatchsep=0.04](17.4,2.7)(18.2,2.7)(17.4,1.9)
\pspolygon[linewidth=0.01,fillstyle=vlines,hatchwidth=0.01,hatchangle=45.0,hatchsep=0.04](17.4,1.5)(18.2,2.3)(18.2,1.5)
\pspolygon[linewidth=0.01,fillstyle=vlines,hatchwidth=0.01,hatchangle=45.0,hatchsep=0.04](18.6,2.7)(18.6,1.9)(19.4,1.9)
\pspolygon[linewidth=0.01,fillstyle=vlines,hatchwidth=0.01,hatchangle=45.0,hatchsep=0.04](17.4,1.1)(18.2,1.1)(17.4,0.3)
\pspolygon[linewidth=0.01,fillstyle=vlines,hatchwidth=0.01,hatchangle=45.0,hatchsep=0.04](17.8,0.3)(18.6,0.3)(18.6,1.1)
\pspolygon[linewidth=0.01,fillstyle=vlines,hatchwidth=0.01,hatchangle=45.0,hatchsep=0.04](19.0,0.7)(19.0,1.5)(19.8,1.5)
\pspolygon[linewidth=0.01,fillstyle=vlines,hatchwidth=0.01,hatchangle=45.0,hatchsep=0.04](19.0,0.3)(19.8,0.3)(19.8,1.1)
\pspolygon[linewidth=0.01,fillstyle=vlines,hatchwidth=0.01,hatchangle=45.0,hatchsep=0.04](20.2,0.3)(21.0,0.3)(20.2,1.1)
\usefont{T1}{ptm}{m}{n}
\rput(22.433594,2.655){\footnotesize $j=9$}
\usefont{T1}{ptm}{m}{n}
\rput(22.463594,2.255){\footnotesize $e=1$}
\usefont{T1}{ptm}{m}{n}
\rput(19.913593,2.195){\footnotesize ${V}_{9}$}
\pscircle[linewidth=0.01,dimen=outer](18.6,1.5){0.1}
\psline[linewidth=0.01cm](17.8,3.5)(17.8,0.3)
\psline[linewidth=0.01cm](18.2,3.1)(18.2,0.3)
\psline[linewidth=0.01cm](18.6,2.7)(18.6,0.3)
\psline[linewidth=0.01cm](19.0,2.3)(19.0,0.3)
\psline[linewidth=0.01cm](19.4,1.9)(19.4,0.3)
\psline[linewidth=0.01cm](19.8,1.5)(19.8,0.3)
\psline[linewidth=0.01cm](20.6,0.7)(20.6,0.3)
\psline[linewidth=0.01cm](17.4,0.7)(20.6,0.7)
\psline[linewidth=0.01cm](17.4,1.1)(20.2,1.1)
\psline[linewidth=0.01cm](17.4,1.5)(19.8,1.5)
\psline[linewidth=0.01cm](17.4,1.9)(19.4,1.9)
\psline[linewidth=0.01cm](17.4,2.3)(19.0,2.3)
\psline[linewidth=0.01cm](17.4,2.7)(18.6,2.7)
\psline[linewidth=0.01cm](17.4,3.5)(17.8,3.5)
\pspolygon[linewidth=0.01](1.2,-1.1)(1.2,-5.1)(5.2,-5.1)
\usefont{T1}{ptm}{m}{n}
\rput(3.8335938,-2.925){\footnotesize ${V}_{10}$}
\pspolygon[linewidth=0.01,fillstyle=vlines,hatchwidth=0.01,hatchangle=45.0,hatchsep=0.08](1.2,-1.1)(1.2,-3.9)(4.0,-3.9)
\pspolygon[linewidth=0.01,fillstyle=vlines,hatchwidth=0.01,hatchangle=45.0,hatchsep=0.04](1.2,-5.1)(1.2,-4.3)(2.0,-4.3)
\pspolygon[linewidth=0.01,fillstyle=vlines,hatchwidth=0.01,hatchangle=45.0,hatchsep=0.04](1.6,-5.1)(2.4,-5.1)(2.4,-4.3)
\pspolygon[linewidth=0.01,fillstyle=vlines,hatchwidth=0.01,hatchangle=45.0,hatchsep=0.04](2.8,-4.3)(3.6,-4.3)(2.8,-5.1)
\pspolygon[linewidth=0.01,fillstyle=vlines,hatchwidth=0.01,hatchangle=45.0,hatchsep=0.04](3.2,-5.1)(4.0,-5.1)(4.0,-4.3)
\pspolygon[linewidth=0.01,fillstyle=vlines,hatchwidth=0.01,hatchangle=45.0,hatchsep=0.04](4.4,-4.3)(4.4,-5.1)(5.2,-5.1)
\psline[linewidth=0.01cm](1.2,-4.3)(4.4,-4.3)
\psline[linewidth=0.01cm](1.2,-4.7)(4.8,-4.7)
\psline[linewidth=0.01cm](1.6,-4.3)(1.6,-5.1)
\psline[linewidth=0.01cm](2.0,-4.3)(2.0,-5.1)
\psline[linewidth=0.01cm](3.2,-4.3)(3.2,-5.1)
\psline[linewidth=0.01cm](3.6,-4.3)(3.6,-5.1)
\psline[linewidth=0.01cm](4.8,-4.7)(4.8,-5.1)
\usefont{T1}{ptm}{m}{n}
\rput(2.1535938,-3.125){\footnotesize ${V}_{7}$}
\usefont{T1}{ptm}{m}{n}
\rput(5.1135936,-2.745){\footnotesize $j=10$}
\usefont{T1}{ptm}{m}{n}
\rput(5.063594,-3.145){\footnotesize $e=0$}
\psline[linewidth=0.01cm](0.0,-6.7)(23.0,-6.7)
\pspolygon[linewidth=0.01,fillstyle=vlines,hatchwidth=0.01,hatchangle=45.0,hatchsep=0.04](12.4,6.1)(12.4,5.3)(13.2,5.3)
\pspolygon[linewidth=0.01](2.2,5.1)(3.2,5.1)(2.2,5.9)
\psline[linewidth=0.01cm](12.8,4.9)(12.8,5.7)
\psline[linewidth=0.01cm](13.2,4.9)(13.2,5.3)
\psline[linewidth=0.01cm](19.6,4.5)(19.6,4.9)
\psline[linewidth=0.01cm](6.8,3.1)(7.2,3.1)
\psline[linewidth=0.01cm](3.8,1.1)(3.8,1.5)
\psline[linewidth=0.01cm](3.4,1.1)(3.4,1.9)
\psline[linewidth=0.01cm](3.0,1.1)(3.0,2.3)
\psline[linewidth=0.01cm](2.6,1.1)(2.6,2.7)
\psline[linewidth=0.01cm](2.2,1.1)(2.2,3.1)
\psline[linewidth=0.01cm](1.8,3.1)(2.2,3.1)
\psline[linewidth=0.01cm](1.8,2.3)(3.0,2.3)
\psline[linewidth=0.01cm](1.8,1.9)(3.4,1.9)
\psline[linewidth=0.01cm](1.8,1.5)(3.8,1.5)
\psline[linewidth=0.01cm](12.4,5.7)(12.8,5.7)
\psline[linewidth=0.01cm](18.8,4.5)(18.8,5.7)
\psline[linewidth=0.01cm](18.4,4.5)(18.4,6.1)
\psline[linewidth=0.01cm](18.0,6.1)(18.4,6.1)
\psline[linewidth=0.01cm](18.0,4.9)(19.6,4.9)
\pspolygon[linewidth=0.01,fillstyle=vlines,hatchwidth=0.01,hatchangle=45.0,hatchsep=0.04](18.0,5.3)(18.0,4.5)(18.8,4.5)
\pspolygon[linewidth=0.01,fillstyle=vlines,hatchwidth=0.01,hatchangle=45.0,hatchsep=0.04](18.4,5.3)(19.2,4.5)(19.2,5.3)
\psline[linewidth=0.01cm](18.2,4.5)(20.2,4.5)
\psline[linewidth=0.01cm](16.0,6.7)(16.0,-6.7)
\usefont{T1}{ptm}{m}{n}
\rput(10.033594,-2.805){\footnotesize $j=11$}
\usefont{T1}{ptm}{m}{n}
\rput(9.983594,-3.205){\footnotesize $e=0$}
\pspolygon[linewidth=0.01](6.0,-1.1)(6.0,-5.5)(10.4,-5.5)
\pspolygon[linewidth=0.01,fillstyle=vlines,hatchwidth=0.01,hatchangle=45.0,hatchsep=0.08](6.0,-1.1)(6.0,-3.9)(8.8,-3.9)
\pspolygon[linewidth=0.01,fillstyle=vlines,hatchwidth=0.01,hatchangle=45.0,hatchsep=0.04](6.0,-4.3)(6.8,-4.3)(6.0,-5.1)
\pspolygon[linewidth=0.01,fillstyle=vlines,hatchwidth=0.01,hatchangle=45.0,hatchsep=0.04](6.0,-5.5)(6.8,-5.5)(6.8,-4.7)
\pspolygon[linewidth=0.01,fillstyle=vlines,hatchwidth=0.01,hatchangle=45.0,hatchsep=0.04](7.2,-4.3)(8.0,-4.3)(7.2,-5.1)
\pspolygon[linewidth=0.01,fillstyle=vlines,hatchwidth=0.01,hatchangle=45.0,hatchsep=0.04](7.2,-5.5)(8.0,-4.7)(8.0,-5.5)
\pspolygon[linewidth=0.01,fillstyle=vlines,hatchwidth=0.01,hatchangle=45.0,hatchsep=0.04](8.4,-4.3)(9.2,-4.3)(8.4,-5.1)
\pspolygon[linewidth=0.01,fillstyle=vlines,hatchwidth=0.01,hatchangle=45.0,hatchsep=0.04](8.4,-5.5)(9.2,-4.7)(9.2,-5.5)
\pspolygon[linewidth=0.01,fillstyle=vlines,hatchwidth=0.01,hatchangle=45.0,hatchsep=0.04](9.6,-4.7)(9.6,-5.5)(10.4,-5.5)
\psline[linewidth=0.01cm](6.0,-4.3)(9.2,-4.3)
\psline[linewidth=0.01cm](6.0,-4.7)(9.6,-4.7)
\psline[linewidth=0.01cm](6.0,-5.1)(10.0,-5.1)
\psline[linewidth=0.01cm](6.4,-4.3)(6.4,-5.5)
\psline[linewidth=0.01cm](6.8,-4.3)(6.8,-5.5)
\psline[linewidth=0.01cm](7.2,-4.3)(7.2,-5.5)
\psline[linewidth=0.01cm](7.6,-4.3)(7.6,-5.5)
\psline[linewidth=0.01cm](8.0,-4.3)(8.0,-5.5)
\psline[linewidth=0.01cm](8.4,-4.3)(8.4,-5.5)
\psline[linewidth=0.01cm](8.8,-4.3)(8.8,-5.5)
\psline[linewidth=0.01cm](9.2,-4.3)(9.2,-5.5)
\psline[linewidth=0.01cm](10.0,-5.1)(10.0,-5.5)
\usefont{T1}{ptm}{m}{n}
\rput(6.913594,-3.205){\footnotesize ${V}_{7}$}
\usefont{T1}{ptm}{m}{n}
\rput(15.433594,-2.805){\footnotesize $j=12$}
\usefont{T1}{ptm}{m}{n}
\rput(15.383594,-3.205){\footnotesize $e=1$}
\pspolygon[linewidth=0.01](10.8,-1.1)(10.8,-5.9)(15.6,-5.9)
\pspolygon[linewidth=0.01,fillstyle=vlines,hatchwidth=0.01,hatchangle=45.0,hatchsep=0.08](10.8,-1.1)(10.8,-3.9)(13.6,-3.9)
\psline[linewidth=0.01cm](10.8,-4.3)(14.0,-4.3)
\psline[linewidth=0.01cm](10.8,-4.7)(14.4,-4.7)
\psline[linewidth=0.01cm](10.8,-5.1)(14.8,-5.1)
\psline[linewidth=0.01cm](11.2,-4.3)(11.2,-5.9)
\psline[linewidth=0.01cm](11.6,-4.3)(11.6,-5.9)
\psline[linewidth=0.01cm](12.0,-4.3)(12.0,-5.9)
\psline[linewidth=0.01cm](12.4,-4.3)(12.4,-5.9)
\psline[linewidth=0.01cm](12.8,-4.3)(12.8,-5.9)
\psline[linewidth=0.01cm](13.2,-4.3)(13.2,-5.9)
\psline[linewidth=0.01cm](13.6,-4.3)(13.6,-5.9)
\psline[linewidth=0.01cm](14.0,-4.3)(14.0,-5.9)
\psline[linewidth=0.01cm](14.8,-5.1)(14.8,-5.9)
\usefont{T1}{ptm}{m}{n}
\rput(11.7135935,-3.205){\footnotesize ${V}_{7}$}
\psline[linewidth=0.01cm](14.4,-4.7)(14.4,-5.9)
\psline[linewidth=0.01cm](15.2,-5.5)(15.2,-5.9)
\psline[linewidth=0.01cm](10.8,-5.5)(15.2,-5.5)
\pspolygon[linewidth=0.01,fillstyle=vlines,hatchwidth=0.01,hatchangle=45.0,hatchsep=0.04](10.8,-5.5)(11.2,-5.1)(11.6,-5.9)(10.8,-5.9)
\pspolygon[linewidth=0.01,fillstyle=vlines,hatchwidth=0.01,hatchangle=45.0,hatchsep=0.04](12.8,-5.9)(13.6,-5.9)(13.2,-5.1)(12.8,-5.5)
\pspolygon[linewidth=0.01,fillstyle=vlines,hatchwidth=0.01,hatchangle=45.0,hatchsep=0.04](14.8,-5.1)(14.8,-5.9)(15.6,-5.9)
\usefont{T1}{ptm}{m}{n}
\rput(22.353594,-2.405){\footnotesize $j=4$}
\usefont{T1}{ptm}{m}{n}
\rput(22.393593,-2.805){\footnotesize $k=1$}
\pspolygon[linewidth=0.01](16.2,-0.1)(16.2,-6.5)(22.6,-6.5)
\pspolygon[linewidth=0.01,fillstyle=vlines,hatchwidth=0.01,hatchangle=45.0,hatchsep=0.08](16.2,-0.1)(16.2,-4.1)(20.2,-4.1)
\psline[linewidth=0.01cm](16.2,-4.5)(20.6,-4.5)
\psline[linewidth=0.01cm](16.2,-4.9)(20.2,-4.9)
\psline[linewidth=0.01cm](16.2,-5.3)(20.2,-5.3)
\psline[linewidth=0.01cm](16.6,-4.5)(16.6,-6.5)
\psline[linewidth=0.01cm](17.0,-4.5)(17.0,-6.5)
\psline[linewidth=0.01cm](17.4,-4.5)(17.4,-6.5)
\psline[linewidth=0.01cm](17.8,-4.5)(17.8,-6.5)
\psline[linewidth=0.01cm](18.2,-4.5)(18.2,-6.5)
\psline[linewidth=0.01cm](18.6,-4.5)(18.6,-6.5)
\psline[linewidth=0.01cm](19.8,-4.5)(19.8,-6.5)
\usefont{T1}{ptm}{m}{n}
\rput(17.393593,-3.005){\footnotesize ${V}_{10}$}
\psline[linewidth=0.01cm](16.2,-5.7)(20.2,-5.7)
\psline[linewidth=0.01cm](16.2,-6.1)(20.2,-6.1)
\pspolygon[linewidth=0.01,fillstyle=vlines,hatchwidth=0.01,hatchangle=45.0,hatchsep=0.08](20.6,-4.5)(20.6,-6.5)(22.6,-6.5)
\pspolygon[linewidth=0.01,fillstyle=vlines,hatchwidth=0.01,hatchangle=45.0,hatchsep=0.04](16.2,-5.7)(16.2,-6.5)(17.0,-6.5)
\pspolygon[linewidth=0.01,fillstyle=vlines,hatchwidth=0.01,hatchangle=45.0,hatchsep=0.04](17.4,-6.5)(17.8,-6.5)(17.8,-6.1)(17.4,-5.7)(17.0,-6.1)
\psline[linewidth=0.01cm](20.2,-4.5)(20.2,-6.5)
\psline[linewidth=0.01cm](19.4,-4.5)(19.4,-6.5)
\psline[linewidth=0.01cm](19.0,-4.5)(19.0,-6.5)
\usefont{T1}{ptm}{m}{n}
\rput(22.383595,-3.205){\footnotesize $e=3$}
\pspolygon[linewidth=0.01,fillstyle=vlines,hatchwidth=0.01,hatchangle=45.0,hatchsep=0.04](18.2,-6.5)(19.0,-6.5)(19.0,-6.1)(18.6,-5.7)
\pspolygon[linewidth=0.01,fillstyle=vlines,hatchwidth=0.01,hatchangle=45.0,hatchsep=0.04](19.4,-6.5)(20.2,-6.5)(20.2,-5.7)
\pspolygon[linewidth=0.01,fillstyle=vlines,hatchwidth=0.01,hatchangle=45.0,hatchsep=0.04](19.4,-6.1)(19.8,-5.7)(19.8,-5.3)(19.4,-5.3)(19.0,-5.7)
\pspolygon[linewidth=0.01,fillstyle=vlines,hatchwidth=0.01,hatchangle=45.0,hatchsep=0.04](20.2,-5.3)(20.2,-4.5)(19.4,-4.5)
\pspolygon[linewidth=0.01,fillstyle=vlines,hatchwidth=0.01,hatchangle=45.0,hatchsep=0.04](19.0,-5.3)(19.4,-4.9)(19.0,-4.5)(18.6,-4.5)(18.6,-4.9)
\pspolygon[linewidth=0.01,fillstyle=vlines,hatchwidth=0.01,hatchangle=45.0,hatchsep=0.04](18.6,-5.3)(18.2,-4.5)(17.4,-4.5)(17.8,-4.9)
\pspolygon[linewidth=0.01,fillstyle=vlines,hatchwidth=0.01,hatchangle=45.0,hatchsep=0.04](18.2,-5.3)(17.4,-4.9)(17.8,-5.7)(18.2,-6.1)
\pspolygon[linewidth=0.01,fillstyle=vlines,hatchwidth=0.01,hatchangle=45.0,hatchsep=0.04](17.0,-4.9)(17.4,-5.3)(17.0,-5.7)(16.6,-5.7)(16.6,-5.3)
\pspolygon[linewidth=0.01,fillstyle=vlines,hatchwidth=0.01,hatchangle=45.0,hatchsep=0.04](16.2,-5.3)(17.0,-4.5)(16.2,-4.5)
\usefont{T1}{ptm}{m}{n}
\rput(21.113594,-5.805){\footnotesize ${V}_{5}$}
\pscircle[linewidth=0.01,dimen=outer](22.58,-6.5){0.1}
\pscircle[linewidth=0.01,dimen=outer](22.18,-6.08){0.1}
\pscircle[linewidth=0.01,dimen=outer](22.18,-6.48){0.1}
\psline[linewidth=0.01cm](6.8,2.5)(7.8,2.5)
\psline[linewidth=0.01cm](8.6,0.7)(8.6,1.7)
\psline[linewidth=0.01cm](6.8,1.7)(7.8,0.7)
\psline[linewidth=0.01cm](7.8,2.5)(7.3,1.22)
\psline[linewidth=0.01cm](7.3,1.2)(8.6,1.68)
\psline[linewidth=0.01cm](11.6,2.7)(12.6,2.7)
\psline[linewidth=0.01cm](13.8,0.5)(13.8,1.5)
\psline[linewidth=0.01cm](11.6,1.5)(12.6,0.5)
\pspolygon[linewidth=0.01,fillstyle=vlines,hatchwidth=0.01,hatchangle=45.0,hatchsep=0.04](12.8,2.5)(12.4,2.1)(12.4,1.7)(12.8,1.7)(13.2,2.1)
\psline[linewidth=0.01cm](12.6,2.7)(11.6,1.5)
\psline[linewidth=0.01cm](12.6,0.5)(13.8,1.5)
\psline[linewidth=0.01cm](13.38,1.92)(11.96,1.14)
\psline[linewidth=0.01cm](17.4,2.9)(18.4,2.9)
\psline[linewidth=0.01cm](20.0,0.3)(20.0,1.3)
\psline[linewidth=0.01cm](20.0,1.3)(18.8,0.3)
\psline[linewidth=0.01cm](18.8,0.3)(18.6,1.5)
\psline[linewidth=0.01cm](19.6,1.7)(18.6,1.5)
\psline[linewidth=0.01cm](17.6,0.3)(18.6,1.5)
\psline[linewidth=0.01cm](17.4,1.3)(18.6,1.5)
\psline[linewidth=0.01cm](18.4,2.9)(18.6,1.5)
\psline[linewidth=0.01cm](17.4,1.7)(18.4,2.9)
\pspolygon[linewidth=0.01,fillstyle=vlines,hatchwidth=0.01,hatchangle=45.0,hatchsep=0.04](11.6,-4.3)(10.8,-4.3)(10.8,-5.1)
\pspolygon[linewidth=0.01,fillstyle=vlines,hatchwidth=0.01,hatchangle=45.0,hatchsep=0.04](12.4,-4.3)(12.4,-5.1)(11.6,-4.7)(12.0,-4.3)
\pspolygon[linewidth=0.01,fillstyle=vlines,hatchwidth=0.01,hatchangle=45.0,hatchsep=0.04](12.8,-4.3)(12.8,-5.1)(13.6,-4.3)
\pspolygon[linewidth=0.01,fillstyle=vlines,hatchwidth=0.01,hatchangle=45.0,hatchsep=0.04](14.0,-4.3)(13.6,-4.7)(13.6,-5.1)(14.0,-5.1)(14.4,-4.7)
\pspolygon[linewidth=0.01,fillstyle=vlines,hatchwidth=0.01,hatchangle=45.0,hatchsep=0.04](14.4,-5.9)(14.4,-5.1)(13.6,-5.5)(14.0,-5.9)
\psline[linewidth=0.01cm](10.8,-5.3)(11.8,-4.3)
\psline[linewidth=0.01cm](14.6,-4.9)(14.6,-5.9)
\psline[linewidth=0.01cm](12.6,-4.3)(12.6,-5.9)
\pspolygon[linewidth=0.01,fillstyle=vlines,hatchwidth=0.01,hatchangle=45.0,hatchsep=0.04](12.4,-5.9)(12.4,-5.5)(12.0,-5.1)(11.6,-5.1)(12.0,-5.9)
\psline[linewidth=0.01cm](11.78,-5.9)(11.3,-4.8)
\psline[linewidth=0.01cm](11.3,-4.78)(12.6,-5.22)
\psline[linewidth=0.01cm](13.8,-4.3)(12.6,-5.3)
\psline[linewidth=0.01cm](13.76,-5.9)(13.14,-4.86)
\psline[linewidth=0.01cm](14.6,-4.9)(13.46,-5.4)
\psline[linewidth=0.01cm](17.2,-4.5)(16.2,-5.5)
\psline[linewidth=0.01cm](16.2,-5.5)(17.2,-6.5)
\psline[linewidth=0.01cm](17.2,-4.5)(18.0,-6.5)
\psline[linewidth=0.01cm](18.74,-5.46)(19.2,-6.5)
\psline[linewidth=0.01cm](19.2,-6.5)(20.2,-5.5)
\psline[linewidth=0.01cm](20.2,-5.5)(19.2,-4.5)
\psline[linewidth=0.01cm](18.76,-5.46)(19.72,-5.02)
\psline[linewidth=0.01cm](18.72,-5.48)(18.02,-6.5)
\psline[linewidth=0.01cm](18.72,-5.46)(17.32,-4.82)
\psline[linewidth=0.01cm](16.7,-6.02)(17.6,-5.5)
\psline[linewidth=0.01cm](18.22,-4.5)(18.74,-5.48)
\end{pspicture} 
}
\end{center}
\caption{}\end{figure}

\end{proof}

\begin{remark}
{\rm
Notice that ${\mathcal{L}_4(3^2)}$ consists of quartics with two triple points and the expected dimension is $2$. This linear system has a fixed part, the double line through the two points and a movable part ${\mathcal{L}_2(1^2)}$ i.e. conics through two points, that has dimension $3$. A simple argument shows that if $d=4$, the linear system ${\mathcal{L}}$ is $-1$--special (we have a $-1$--curve, line connecting the 2 points, splitting off twice) and therefore special.\\
\indent
One could mention that case $d=4$ is also a special case for the double points interpolation problem since ${\mathcal{L}_4(2^5)}$ is expected to be empty but it consists of the double conic determined by the $5$ general points.
}
\end{remark}

\end{document}